\documentclass[11pt]{amsart}
\usepackage[all]{xy}
\usepackage {fancybox}
\usepackage{graphicx}
\usepackage{ascmac}
\usepackage{tikz-qtree}
\usepackage{forest}
\usepackage{mathrsfs}
\usepackage{amscd,verbatim}
\usepackage{floatflt}
\usepackage{amssymb}
\usepackage{tikz-cd}

\newtheorem{lemma}{Lemma}[section]

\newtheorem{thm}[lemma]{Theorem}

\newtheorem{prop}[lemma]{Proposition}

\newtheorem{cor}[lemma]{Corollary}

\newtheorem{conj}[lemma]{Conjecture}

\newtheorem{defn}[lemma]{Definition}

\newtheorem{remark}[lemma]{Remark}

\numberwithin{equation}{section}

\begin{document}

\title[$p$-adic Hodge theory for non-commutative algebraic varieties]{Crystalline representations and $p$-adic Hodge theory for non-commutative algebraic varieties}









\newcommand\rightthreearrow{%
        \mathrel{\vcenter{\mathsurround0pt
                \ialign{##\crcr
                        \noalign{\nointerlineskip}$\rightarrow$\crcr
                        \noalign{\nointerlineskip}$\rightarrow$\crcr
                        \noalign{\nointerlineskip}$\rightarrow$\crcr
                }%
        }}%
}

\newcommand{\ulMCorf}{{\ulMCor}^{\text{fin}}}
\newcommand{\ulMMO}{\textbf{\underline{M}}\sO}
\newcommand{\ulMO}{\ul{M}\sO}
\newcommand{\MpU}{(\ol{U},U^{\infty})}
\newcommand{\un}{\mathbf{1}}
\newcommand{\A}{\mathbf{A}}
\renewcommand{\AA}{\mathbb{A}}
\newcommand{\mbA}{\mathbb{A}^1}
\newcommand{\mbAm}{\mathbb{A}^m}
\newcommand{\deAmi}{(\mathbb{A}^{1},d_{i_{1}}\{0\}) \otimes (\mathbb{A}^{1},d_{i_{2}}\{0\}) \otimes \cdots \otimes (\mathbb{A}^{1},\emptyset)}
\newcommand{\deAmoi}{(\mathbb{A}^{1},d_{i_{1}}\{0\}) \otimes (\mathbb{A}^{1},d_{i_{2}}\{0\}) \otimes \cdots \otimes \{0\}}
\newcommand{\deAm}{(\mathbb{A}^{1},d_{{1}}\{0\}) \otimes (\mathbb{A}^{1},d_{{2}}\{0\}) \otimes \cdots \otimes (\mathbb{A}^{1},\emptyset)}
\newcommand{\deAmo}{(\mathbb{A}^{1},d_{{1}}\{0\}) \otimes (\mathbb{A}^{1},d_{{2}}\{0\}) \otimes \cdots \otimes \{0\}}
\newcommand{\deAl}{(\mathbb{A}^{1},d_{1}\{0\})\otimes (\mathbb{A}^{1},d_{2}\{0\})\otimes \cdots \otimes (\mathbb{A}^{1},d_{m-1}\{0\})}
\newcommand{\BB}{\mathbb{B}}
\newcommand{\F}{\mathbf{F}}
\newcommand{\G}{\mathbf{G}}
\renewcommand{\P}{\mathbf{P}}
\renewcommand{\L}{\mathbb{L}}
\newcommand{\N}{\mathbb{N}}
\newcommand{\PP}{\mathbb{P}}
\newcommand{\Q}{\mathbb{Q}}
\newcommand{\Z}{\mathbb{Z}}
\newcommand{\sA}{\mathcal{A}}
\newcommand{\sB}{\mathcal{B}}
\newcommand{\sC}{\mathcal{C}}
\newcommand{\sD}{\mathcal{D}}
\newcommand{\sE}{\mathcal{E}}
\newcommand{\sF}{\mathcal{F}}
\newcommand{\sG}{\mathcal{G}}
\newcommand{\sH}{\mathcal{H}}
\newcommand{\sI}{\mathcal{I}}
\newcommand{\sJ}{\mathcal{J}}
\newcommand{\sK}{\mathcal{K}}
\newcommand{\sL}{\mathcal{L}}
\newcommand{\sO}{\mathcal{O}}
\newcommand{\sP}{\mathcal{P}}
\newcommand{\sQ}{\mathcal{Q}}
\newcommand{\sR}{\mathcal{R}}
\newcommand{\sS}{\mathcal{S}}
\newcommand{\sT}{\mathcal{T}}
\newcommand{\sU}{\mathcal{U}}
\newcommand{\sV}{\mathcal{V}}

\newcommand{\Zt}{\Ztr}
\newcommand{\Shv}{{\operatorname{Shv}}}
\newcommand{\sX}{\mathcal{X}}
\newcommand{\sY}{\mathcal{Y}}
\newcommand{\sZ}{\mathcal{Z}}
\newcommand{\bH}{\mathbb{H}}
\newcommand{\bN}{\mathbb{N}}
\newcommand{\bZ}{\mathbb{Z}}
\newcommand{\fm}{\mathfrak{m}}
\newcommand{\fM}{\mathfrak{M}}
\newcommand{\fU}{\mathfrak{U}}
\newcommand{\fV}{\mathfrak{V}}
\newcommand{\fW}{\mathfrak{W}}
\newcommand{\fX}{\mathfrak{X}}
\newcommand{\fY}{\mathfrak{Y}}
\newcommand{\fZ}{\mathfrak{Z}}
\newcommand{\fn}{\mathfrak{n}}
\newcommand{\fp}{\mathfrak{p}}
\newcommand{\fq}{\mathfrak{q}}
\newcommand{\fr}{\mathfrak{r}}
\newcommand*{\xdashrightarrow}[2][]{%
  \mathrel{%
    \mathpalette{\da@xarrow{#1}{#2}{}\da@rightarrow{\,}{}}{}%
  }%
}
\newcommand{\Xb}{{\overline{X}}}
\newcommand{\Tb}{\overline{T}}
\newcommand{\alphab}{\ol{\alpha}}
\newcommand{\deltab}{\ol{\delta}}
\newcommand{\tilO}{\tilde{\mathcal{O}}}
\newcommand{\Mfget}{\ulM^{\text{h}}_{\text{gfin}}}
\newcommand{\Mhgf}{\ulM^{\text{h}}_{\text{gfin}}}
\newcommand{\Mqfh}{\ulM \text{qfh}}

\newcommand{\tilOmega}{\tilde{\Omega}}
\newcommand{\Ga}{{G^{(0)}}}
\newcommand{\Gb}{{G^{(1)}}}
\newcommand{\CuXY}{\sC_{(\Xb,Y)}}
\newcommand{\ha}{{h^{(0)}}}
\newcommand{\hb}{{h^{(1)}}}
\newcommand{\Mod}{\operatorname{Mod}\hbox{--}}
\newcommand{\Span}{\operatorname{\mathbf{Span}}}
\newcommand{\Cor}{\operatorname{\mathbf{Cor}}}
\newcommand{\Mack}{\operatorname{\mathbf{Mack}}}
\newcommand{\HI}{\operatorname{\mathbf{HI}}}
\newcommand{\Pre}{{\operatorname{\mathbf{Pre}}}}
\newcommand{\Rec}{{\operatorname{\mathbf{Rec}}}}
\newcommand{\SHI}{\operatorname{\mathbf{SHI}}}
\newcommand{\proj}{\operatorname{proj}}

\newcommand{\Cone}{\operatorname{Cone}}
\newcommand{\Ext}{\operatorname{Ext}}
\newcommand{\ul}[1]{{\underline{#1}}}

\newcommand{\Set}{{\operatorname{\mathbf{Set}}}}
\newcommand{\PST}{{\operatorname{\mathbf{PST}}}}
\newcommand{\PS}{{\operatorname{\mathbf{PS}}}}
\newcommand{\NST}{\operatorname{\mathbf{NST}}}
\newcommand{\RS}{\operatorname{\mathbf{RS}}}
\newcommand{\Chow}{\operatorname{\mathbf{Chow}}}
\newcommand{\DM}{\operatorname{\mathbf{DM}}}
\newcommand{\DMR}{\operatorname{\mathbf{DMR}}}
\newcommand{\DR}{\operatorname{\mathbf{DR}}}
\newcommand{\MDM}{\operatorname{\mathbf{MDM}}}
\newcommand{\MDMgm}{\MDM_{\text{gm}}^{\text{eff}}}
\newcommand{\ulMDM}{\operatorname{\mathbf{\underline{M}DM}}}
\newcommand{\Map}{\operatorname{Map}}
\newcommand{\Hom}{\operatorname{Hom}}
\newcommand{\uHom}{\operatorname{\underline{Hom}}}
\newcommand{\uExt}{\operatorname{\underline{Ext}}}
\newcommand{\RHom}{\operatorname{R\underline{Hom}}}
\newcommand{\que}[1]{{\color{green}#1}}
\newcommand{\uG}{{\underline{G}}}
\newcommand{\Ker}{\operatorname{Ker}}
\newcommand{\IM}{\operatorname{Im}}
\renewcommand{\Im}{\operatorname{Im}}
\newcommand{\Coker}{\operatorname{Coker}}
\newcommand{\MpN}{(\ol{N},N^{\infty})}
\newcommand{\Coim}{\operatorname{Coim}}
\newcommand{\Tr}{\operatorname{Tr}}
\newcommand{\Div}{\operatorname{Div}}
\newcommand{\Pic}{\operatorname{Pic}}
\newcommand{\Br}{\operatorname{Br}}
\newcommand{\Spec}{\operatorname{Spec}}
\newcommand{\Proj}{\operatorname{Proj}}
\newcommand{\Sm}{\operatorname{\mathbf{Sm}}}
\newcommand{\Sch}{\operatorname{\mathbf{Sch}}}
\newcommand{\Ab}{\operatorname{\mathbf{Ab}}}
\newcommand{\oo}{\operatornamewithlimits{\otimes}\limits}
\newcommand{\by}{\xrightarrow}
\newcommand{\yb}{\xleftarrow}
\newcommand{\iso}{\by{\sim}}
\newcommand{\osi}{\yb{\sim}}
\newcommand{\rec}{{\operatorname{rec}}}
\newcommand{\pro}[1]{\text{\rm pro}_{#1}\text{\rm--}}
\newcommand{\tr}{{\operatorname{tr}}}
\newcommand{\proper}{{\operatorname{prop}}}
\newcommand{\dom}{{\operatorname{dom}}}
\newcommand{\rat}{{\operatorname{rat}}}
\newcommand{\sat}{{\operatorname{sat}}}
\newcommand{\eff}{{\operatorname{eff}}}
\newcommand{\fin}{{\operatorname{fin}}}
\renewcommand{\o}{{\operatorname{o}}}

\newcommand{\op}{{\operatorname{op}}}
\newcommand{\norm}{{\operatorname{norm}}}
\newcommand{\red}{{\operatorname{red}}}
\newcommand{\cont}{{\operatorname{cont}}}
\newcommand{\Zar}{{\operatorname{Zar}}}
\newcommand{\Nis}{{\operatorname{Nis}}}
\newcommand{\et}{{\operatorname{\acute{e}t}}}
\newcommand{\tto}{\dashrightarrow}
\newcommand{\inj}{\hookrightarrow}
\newcommand{\Inj}{\lhook\joinrel\longrightarrow}
\newcommand{\surj}{\rightarrow\!\!\!\!\!\rightarrow}
\newcommand{\Surj}{\relbar\joinrel\surj}
\newcommand{\Res}{\operatorname{Res}}
\newcommand{\Jac}{{\operatorname{Jac}}}
\newcommand{\id}{{\operatorname{Id}}}
\newcommand{\li}{{\operatorname{li}}}
\newcommand{\divi}{{\operatorname{div}}}
\newcommand{\adm}{{\operatorname{adm}}}
\newcommand{\Supp}{{\operatorname{Supp}}}
\newcommand{\pd}{{\partial}}
\newcommand{\pmods}[1]{\; (\operatorname{mod}^*#1)}
\newcommand{\codim}{{\operatorname{codim}}}
\newcommand{\ch}{{\operatorname{ch}}}
\newcommand{\Sym}{{\operatorname{Sym}}}
\newcommand{\Tot}{{\operatorname{Tot}}}
\newcommand{\CH}{{\operatorname{CH}}}
\newcommand{\Mip}{M^\infty_+}
\newcommand{\M}{\mathbf{M}}
\newcommand{\Mb}{{\overline{M}}}
\newcommand{\Nb}{{\overline{N}}}
\newcommand{\Lb}{{\overline{L}}}
\newcommand{\Zb}{{\overline{Z}}}
\newcommand{\CbS}{{\overline{C}_S}}
\newcommand{\Cbeta}{{\overline{C}_\eta}}
\newcommand{\Gu}{{\underline{G}}}
\def\indlim#1{\underset{{\underset{#1}{\longrightarrow}}}{\mathrm{lim}}\; }
\def\rmapo#1{\overset{#1}{\longrightarrow}}

\newcommand{\Cu}{\mathcal{C}}
\newcommand{\CuX}{\Cu_\Xb}
\newcommand{\CuMN}{\Cu_{(\Mb,N)}}
\newcommand{\ulM}{\underline{\M}}
\newcommand{\can}{\operatorname{can}}
\newcommand{\bC}{\mathbb{C}}
\newcommand{\mf}{\mathfrak}
\newtheorem{notation}[lemma]{Notation}

\newcommand{\Cat}{\operatorname{Cat_{\infty}^{perf}}}
\newcommand{\perf}{\operatorname{perf}}

\newcommand{\GNS}{\operatorname{GNSch}}
\newcommand{\Catsat}{\operatorname{Cat_{\infty,sat}^{perf}}}
\newcommand{\Catcsat}{\operatorname{Cat_{\infty,sat}^{c,perf}}}

\newcommand{\mS}{\mathfrak{S}}
\newcommand{\mStw}{\mathfrak{S}^{(2)}}
\newcommand{\mSth}{\mathfrak{S}^{(3)}}
\newcommand{\mSsttw}{\mathfrak{S}_{\operatorname{st}}^{(2)}}
\newcommand{\mSstth}{\mathfrak{S}_{\operatorname{st}}^{(3)}}

\newcommand{\Sp}{\operatorname{Sp}}
\newcommand{\mm}{\mathfrak{m}}
\newcommand{\nn}{\mathfrak{n}}

\newcommand{\cry}{\operatorname{cry}}
\numberwithin{equation}{section}

\newcommand{\ulMP}{\operatorname{\mathbf{\underline{M}Sm}}}
\newcommand{\ulMNS}{\operatorname{\mathbf{\underline{M}NS}}}
\newcommand{\ulMPS}{\operatorname{\mathbf{\underline{M}PS}}}
\newcommand{\ulMPST}{\operatorname{\mathbf{\underline{M}PST}}}
\newcommand{\ulMNST}{\operatorname{\mathbf{\underline{M}NST}}}
\newcommand{\ulMCor}{\operatorname{\mathbf{\underline{M}Cor}}}
\newcommand{\ulMSm}{\operatorname{\mathbf{\underline{M}Sm}}}
\newcommand{\ulPSCH}{\operatorname{\mathbf{\underline{P}Sch}}}
\newcommand{\ulMSCH}{\operatorname{\mathbf{\underline{M}Sch}}}
\newcommand{\ulPVar}{\operatorname{\mathbf{\underline{P}Var}}}
\newcommand{\ulMVar}{\operatorname{\mathbf{\underline{M}Var}}}
\newcommand{\ulMVarCor}{\operatorname{\mathbf{\underline{M}VarCor}}}
\newcommand{\ulMSmCor}{\operatorname{\mathbf{\underline{M}SmCor}}}
\newcommand{\ulPSm}{\operatorname{\mathbf{\underline{P}Sm}}}

\renewcommand{\lim}{\operatornamewithlimits{\varprojlim}}
\newcommand{\colim}{\operatornamewithlimits{\varinjlim}}
\newcommand{\hocolim}{\operatorname{hocolim}}
\newcommand{\ol}{\overline}
\newcommand{\holim}{\operatorname{holim}}

\newcommand{\Bcry}{B_{\operatorname{crys}}}
\newcommand{\Bst}{B_{\operatorname{st}}}
\newcommand{\Acry}{A_{\operatorname{cry}}}
\newcommand{\Ast}{A_{\operatorname{st}}}
\newcommand{\Ainf}{A_{\operatorname{inf}}}
\newcommand{\Ainfu}{A_{\operatorname{inf}}[\frac{1}{[\epsilon]-1}]}
\newcommand{\Ainfmu}{A_{\operatorname{inf}}[\frac{1}{ \mu }]}

\newcommand{\bcube}{{\ol{\square}}}
\newcommand{\cube}{\square}
\newcommand{\MpMnZ}{\ol{M},\Minf + n\ol{Z}}
\newcommand{\MpNnZ}{\ol{N},N^\inf + nf^{-1}\ol{Z}}
\def\bN{\mathbb{N}}
\def\bZ{\mathbb{Z}}
\def\bQ{\mathbb{Q}}
\def\Ztr{\bZ_\tr}
\def\Qtr{\bQ_\tr}
\def\soO{\overline{\sO}}
\def\Image{\mathrm{Image}}
\def\rC#1{rel\sC(#1)}
\def\pb{\overline{p}}
\def\Xd#1{X_{(#1)}}
\def\Xbd#1{\Xb_{(#1)}}
\def\Ud#1{U_{(#1)}}
\def\isom{\overset{\cong}{\longrightarrow}}

\newcommand{\MpM}{(\ol{M},\Minf)}
\newcommand{\MpZ}{(\ol{Z},\Zinf)}
\newcommand{\ulomega}{\underline{\omega}}

\newcommand{\THH}{\operatorname{THH}}
\newcommand{\TP}{\operatorname{TP}}
\newcommand{\TC}{\operatorname{TC}}
\newcommand{\TCn}{\operatorname{TC}^{-}}
\newcommand{\HP}{\operatorname{HP}}

\def\Zinf{Z^\infty}
\def\Xinf{X^\infty}
\def\Minf{M^{\infty}}
\newcommand{\MCorls}{{\MCor}_{\text{ls}}}
\newcommand{\ulMCorls}{{\ulMCor}_{\text{ls}}}
\newcommand{\ulMSmls}{{\ulMSm}_{\text{ls}}}
\newcommand{\ulMSmlsCor}{{\ulMSm}_{\text{ls}}\Cor}
\newcommand{\udo}{\underline{\omega}}
\newcommand{\bS}{\mathbb{S}}
\newcommand{\LK}{{L_{K(1)}}}

\author{Keiho Matsumoto}

\begin{abstract}
In this paper, we study $p$-adic Hodge theory for non-commutative algebraic varieties. Firstly, we propose a conjecture that for $K$ a complete discretely valued nonarchimedean extension $K$ of $\mathbb{Q}_p$ with perfect residue field $k$ and $\mathcal{T}$ an $\mathcal{O}_K$-linear idempotent-complete, small smooth proper stable $\infty$-category, there exists a $B_{\mathrm{crys}}$-coefficients isomorphism preserving additional structures between the $K(1)$-local $K$-theory on the generic fiber of $\mathcal{T}$ and the topological periodic cyclic homology on the special fiber. This conjecture can be regarded as a non-commutative analog of the crystalline comparison theorem. We then proceed to prove the following results: the topological negative cyclic homology $\pi_i\TCn(\sT/\bS[z];\Z_p)$ admits a Breuil-Kisin module structure, and the non-commutative analog of Bhatt-Morrow-Scholze’s comparison theorems holds. Additionally, we demonstrate that the $\mathbb{Z}_p[G_K]$-module obtained from the topological negative cyclic homology is a $\mathbb{Z}_p$-lattice of a crystalline representation. Finally, we show that when the generic fiber of $\mathcal{T}$ admits a geometric realization in the sense of Orlov, the non-commutative analog of the crystalline comparison theorem proposed by the author holds.
\end{abstract}

\maketitle


    \section{Introduction}

    The aim of this paper is to study $p$-adic Hodge theory for non-commutative algebraic varieties. Specifically, we propose a conjecture and prove several results related to the relationship between $K(1)$-local $K$-theory and topological periodic homology in this non-commutative setting.

\begin{notation}
Fix a prime $p$. Let $K$ be a complete discretely valued nonarchimedean extension $K$ of $\mathbb{Q}_p$ with perfect residue
field $k$. Here, $\sO_K$ is the ring of integers of $K$ and $\pi\in \sO_K$ is a uniformizer. We write $\Cu$ to denote the completion $\hat{\overline{K}}$ of $\overline{K}$ endowed with its unique absolute value extending the given absolute value on $K$, let $W$ be the Witt ring of $k$, and let $K_0$ be the fraction field of $W$. Let $\mf{m}$ be the maximal ideal of $\sO_K$. For a spectrum $S$, let $L_{K(1)}S$ be the Bousfield localization of complex $K$ theory at prime $p$.
\end{notation}

Cohomology theories such as de Rham cohomology $R\Gamma_{\text{dR}}(-/K)$, Hodge cohomology $R\Gamma_{\Zar}(-,\Omega^*_{-/K})$, $l$-adic cohomology $R\Gamma_{\et}(-,\Z_l)$ and crystalline cohomology $R\Gamma_{\text{cry}}(-/W(k))$ are important tools in the study of algebraic geometry and arithmetic geometry. On the other hand, homological invariants such as (topological) periodic homology, (topological) cyclic homology, and $K$ theory play an important role in the study of $C^*$-algebra, Lie algebra and non-commutative geometry. There are deep and subtle links between cohomology invariants and homological invariants. One of the most well-known examples is the Atiyah-Hirzebruch spectral sequence. For a finite-dimensional CW-complex $M$, Atiyah-Hirzebruch \cite{AHss} prove that there is a spectral sequence:
\begin{equation}\label{AHsss}
E_2^{i,j}=\left\{
\begin{array}{ll}
H^{i}_{\text{Sing}}(M,\Z) & j \text{ even} \\
0 & j \text{ odd} 
\end{array}
\right. \Longrightarrow K^{\text{top}}_{j-i}(M).
\end{equation}
The Thomason spectral sequence is an arithmetic-geometrical analogue
of the Atiyah-Hirzebruch spectral sequence. For a smooth variety $X$ over a field of characteristic $0$, Thomason \cite{Thomason} shows that there exists a spectral sequence
\begin{equation}\label{bakido}
E_2^{i,j}=\left\{
\begin{array}{ll}
H^{i}_\et(X,\Z_p(l)) & j = 2l \\
0 & j \text{ odd} 
\end{array}
\right. \Longrightarrow \pi_{j-i} L_{K(1)}K(X).
\end{equation} 
and this spectral sequence degenerates after tensoring $\mathbb{Q}_p$. Besides, Hesselholt \cite{Hes96} shows a close relation between $p$-adic cohomology theory and topological cyclic homology. 

Bondal-Kapranov, Orlov \cite{noncommutativescheme} and Kontsevich \cite{KOntsevichnoncommutative} introduce \textit{non-commutative algebraic geometry} in which a dg-category (or a stable $\infty$-category) is studied as a \textit{non-commutative space}. Nowadays, non-commutative algebraic geometry plays an important role in research on mirror symmetry, mathematical physics and algebraic geometry. Besides, homological invariants are well-defined for dg-categories and stable $\infty$-categories. It has been known that some comparison theorems between cohomology groups can be formulated naturally for dg-categories and stable $\infty$-categories: instead of cohomology theories, one can consider homological invariants. For a smooth proper dg-category $\sT$ over $\mathbb{C}$, Kaledin \cite{Kaledin} proves that there is an isomorphism $\text{HP}_n(\sT/\mathbb{C}) \simeq \bigoplus_{i\in \Z} \text{HH}_{n+2i}(\sT/\mathbb{C})$, which has been conjectured by Kontsevich-Soibelman \cite{KontseS} and can be regarded as a \textit{non-commutative Hodge decomposition} via Connes \cite{Conne}, Feigin–Tsygan \cite{Feigin} and Hochschild-Kostant-Rosenberg \cite{HKR}. Besides, Blanc \cite{Blanc} conjectured that there is an equivalence $\operatorname{Ch}^{\operatorname{top}}\wedge_{\bS}H\bC:K_{\operatorname{top}}(\sT)\wedge_{\bS}H\bC \to \HP(\sT/\bC)$, which can be regarded as the \textit{non-commutative de Rham comparison theorem}, and in
some cases, the equivalence is proved by A. Kahn \cite{khan2023lattice}. For a smooth proper stable $\infty$-category $\sT$ over $W$, Scholze proves that there is an isomorphism of $W$-modules $\pi_n \TP(\sT_{k};\Z_p) \simeq \pi_i \HP(\sT/W)$ (it has been unpublished yet), and Petrov-Vologodsky obtain the same result for any stable $\infty$-category \cite{petrov}.

In the study of $p$-adic cohomology theories, crystalline comparison theory \cite{Tsujicrys}, \cite{Falcrys} states that for a smooth proper variety $X$ over $\sO_K$, the $p$-adic \'etale cohomology $H^i_\et(X_\sC,\Q_p)\otimes_{\Q_p}\Bcry$ is isomorphic to $H^i_{cry}(X_k/W)\otimes_{W}\Bcry$, and the isomorphism is compatible with $G_K$-action, Frobenius endomorphism and filtration. We study a non-commutative version of the crystalline comparison theorem. For a commutative ring $R$, we will refer to $R$-linear idempotent-complete, small stable $\infty$-categories simply as $R$-linear categories. For an $\sO_{K}$-linear category $\sT$, $G_K$ acts continuously on $\Z_p$-module $\pi_i \LK K(\sT_\sC)$, and there is a Frobenius operator $\text{Fr}:\pi_i\TP(\sT_k;\Z_p)[\frac{1}{p}] \overset{\can}{\simeq}\pi_i\TCn(\sT_k;\Z_p)[\frac{1}{p}] \overset{\varphi}{\to}\pi_i\TP(\sT_k;\Z_p)[\frac{1}{p}]$ (see \cite{KunnethTP}). Inspired by Petrov and Vologodsky's work on non-commutative crystalline cohomology theory \cite{petrov}, and the study of motivic filtration of the $K(1)$-local $K$ theory \cite{Thomason} and the topological periodic homology \cite{BMS2}, we predict the following conjecture.
\begin{conj}[Non-commutative version of crystalline comparison theorem~\cite{Tsujicrys}, \cite{Falcrys}] \label{NCcryconj}
Let $\sT$ be a smooth proper $\sO_K$-linear category. Then there is an isomorphism of $\Bcry$-module:
\[
\pi_i\TP(\sT_k;\Z_p)\otimes_{W} \Bcry \simeq \pi_i \LK K(\sT_\sC)\otimes_{\Z_p}\Bcry
\]
which is compatible with $G_K$-action and Frobenius endomorphism. In particular, the $p$-adic representation $\pi_i\LK K(\sT_\sC)\otimes_{\Z_p}\Q_p$ is crystalline.
\end{conj}
\begin{remark}
Assume $\sO_K=W$.
Then the isomorphism of $W$-modules
\[
\pi_n \TP(\sT_{k};\Z_p)\;\xrightarrow{\;\sim\;}\;\pi_n \HP(\sT/W)
\]
has been proved independently by Scholze (unpublished) and by Petrov–Vologodsky \cite{petrov}
for $p>2$. 
For all primes $p$ and for smooth proper dg-categories, this result 
has since been established by Devalapurkar–Raksit~\cite{Raksit} and Mao~\cite{Mao}. 
Consequently, the non-commutative Hodge--de~Rham filtration on $\pi_n\HP(\sT/W)$ 
induces a natural filtration on $\TP(\sT_{k};\Z_p)$. 
It is further expected that for a general complete DVR $\sO_K$ of mixed characteristic, 
the base change
\[
\pi_n\TP(\sT_{k};\Z_p)\otimes_W K
\]
inherits a filtration arising from the non-commutative Hodge--de~Rham filtration on 
$\pi_n\HP(\sT/\sO_K)$. 
Moreover, we expect the isomorphism in Conjecture~1.2 to be compatible with these filtrations.
\end{remark}

For a smooth proper variety $X$ over $\sO_K$, by \cite{Thomason} one has a $G_K$-equivariant isomorphism $\pi_i\LK K(\sT_\sC)\otimes_{\Z_p} \Q_p \simeq \bigoplus_{n\in\Z} H^{i+2n}_{\et}(X_{\Cu},\Z_p(n))\otimes_{\Z_p}\Q_p$, and similarly by \cite{BMS1} one has an isomorphism $\pi_i\TP(\sT_k;\Z_p)[\frac{1}{p}] \simeq \bigoplus_{n\in\Z} H^{i+2n}_{\cry}(X_{k}/W[\frac{1}{p}])(n)$ of isocrystals, thus the conjecture holds for $\sT=\perf(X)$ via crystalline comparison theorem~\cite{Tsujicrys}. In this paper, we approach this conjecture via $K$-theoritical version of Bhatt-Morrow-Scholze's comparison theorem \cite{BMS2}. We will use the language of stable $\infty$-categories, following Lurie \cite{Lurie17}.
\begin{defn}
 For an $\mathbb{E}_\infty$-ring $R$, we let $\Cat (R)$ denote the $\infty$-category of $R$-linear categories, where the morphisms are exact functors.
\end{defn}

\begin{defn}
For an $\mathbb{E}_\infty$-ring $R$, we let $\Catsat (R)$ denote the $\infty$-category of smooth proper $R$-linear categories, where the morphisms are exact functors. 
\end{defn}
\begin{defn}[{See also {\cite[section~4.1]{noncommutativescheme}}}]\label{geomreal}
    For an $\mathbb{E}_\infty$-ring $R$, an $R$-linear category $\sT$ \textit{admits 
a geometric realization} if there is a derived scheme $X$ over $R$ such that the truncation $\pi_0(X)$ is separated scheme of finite type over $\pi_0(R)$ and there is a fully faithful admissible inclusion $\sT \subset \perf(X)$ of $R$-linear categories (see \cite[Definition 3.1]{admissiblesub}). 
\end{defn} 
In many cases, a stable $\infty$-category is known to admit a geometric realization. For example, the derived Fukaya category of a symplectic manifold is known or expected to admit a geometric realization from the study of mirror symmetry. Besides, Lunts-Bergh-Schn\"urer \cite{geometricity} proved that the stable infinity category of perfect complexes on a smooth proper Deligne-Mumford admits a geometric realization.
\begin{remark}
Assume $R$ is an algebraic closed field of characteristic $0$. In \cite[Question 4.4]{noncommutativescheme}, Orlov asked if there exists $R$-linear idempotent-complete, small smooth proper stable $\infty$-categories which don't admit a geometric realization. This is still an important open problem. 
\end{remark}

\begin{remark}
 If a smooth proper $R$-linear category $\sT$ admits a geometric realization $\sT \hookrightarrow \perf(X)$, the dual $\sT^{\text{op}} \in \Catsat$ also admits a geometric realization $\sT^{\text{op}} \hookrightarrow \perf(X)^{\text{op}} =\perf(X)$.
\end{remark}

 Fix a sequence of elements $\zeta_n\in \ol{K}$ inductively such that $\zeta_0=1$ and $(\zeta_{n+1})^p=\zeta_n$, let $\epsilon=(\zeta_0,\zeta_1,\zeta_2,...)\in \lim_{\text{Frob}}\sO_\Cu/p$, and let $[\epsilon]\in \Ainf$ be the Teichm\"uller lift. Let $\mS=W[[z]]$, and let $\tilde{\theta}:\mS \to \sO_K$ be the usual map whose kernel is generated by Eisenstein polynomial $E$ of $\pi$. Let $\phi:\mS \to \Ainf$ be the $W$-linear map that sends $z$ to $[\pi^\flat]$. We write $\xi= \phi(E)$. We write $\mu=[\epsilon]-1$. Let $\varphi:\mS \to \mS$ be a Frobenius endomorphism which is Frobenius on $W$ and sends $z$ to $z^p$. We prove a non-commutative version of Bhatt-Morrow-Scholze's Breuil-Kisin cohomology theory $R\Gamma_{\mS}(-)$ \cite{BMS2}. 

\begin{thm}[{Theorem~\ref{NCBMS}}, {Non-commutative version of \cite{BMS2}}]\label{NCBMSintro}
  Let $\sT$ be an $\sO_K$-linear smooth proper category. Then there is a natural number $n$ such that the following holds:
 \begin{itemize}
   \item[(1)] For any $i\geq n$, $\pi_i \TCn(\sT/\bS[z];\Z_p)$ has a natural structure of a Breuil-Kisin module.
   \item[(1)'] For any $i\geq n$, $\pi_i \TCn(\sT/\bS[z];\Z_p)^{\vee}$ has a natural structure of a Breuil-Kisin module of finite $E$-height.
   \item[(2)] (K(1)-local K theory comparison) Assume $\sT_{\Cu}$ admits a geometric realization. For any $i\geq n$, after scalar extension along $\ol{\phi}:\mS\to \Ainf$ which sends $z$ to $[\pi^{\flat}]^p$ and is the Frobenius on $W$, one recovers K(1)-local K theory of the generic fiber
   \[
   \pi_i \TCn(\sT/\bS[z];\Z_p) \otimes_{\mS,\ol{\phi}} \Ainf[\frac{1}{\mu}]^{\wedge}_p \simeq \pi_i\LK K(\sT_{\Cu})\otimes_{\Z_p} \Ainf[\frac{1}{\mu}]^{\wedge}_p
   \]
   \item[(3)] (topological periodic homology theory comparison) For any $i\geq n$, after scalar extension along the map $\tilde{\phi}:\mS\to W$ which is the Frobenius on $W$ and sends $z$ to $0$, one recovers topological periodic homology theory of the special fibre:
  \[
   \pi_i \TCn(\sT/\bS[z];\Z_p)[\frac{1}{p}] \otimes_{\mS[\frac{1}{p}],\tilde{\phi}} K_0 \simeq \pi_i \TP(\sT_k;\Z_p)[\frac{1}{p}].
   \]
  \end{itemize}
\end{thm}

  Bhatt-Morrow-Scholze's Breuil-Kisin cohomology theory \cite[Theorem~1.2]{BMS2} implies crystalline comparison theorem. On the other hand, Theorem~\ref{NCBMSintro} does not imply Conjecture~\ref{NCcryconj}. This difference arises as follows. On Breuil-Kisin cohomology theory, there is a $G_K$-equivariant isomorphism (see \cite[Theorem 1.8 (iii)]{BMS1})
  \begin{equation}\label{diffCNC}
  R\Gamma_{\Ainf}(X_{\sO_\Cu})\otimes_{\Ainf} \Acry \simeq R\Gamma_{\cry}(X_{\sO_\Cu/p}/\Acry).
  \end{equation}
Combine \eqref{diffCNC} with \cite[Theorem~1.8 (iv)]{BMS1}, there is canonical $(G_K,\varphi)$-equivariant isomorphism 
\begin{equation}\label{BMSmulocal}
R\Gamma_{\et}(X_{\Cu},\Z_p)\otimes_{\Z_p} \Acry[\frac{1}{p\mu}] \simeq R\Gamma_{\cry}(X_{\sO_\Cu/p}/\Acry)[\frac{1}{p\mu}].
\end{equation}
The isomorphism induces the crystalline comparison theorem (see \cite[Theorem 14.4]{BMS1}). On the other hand, the non-commutative analogue of \eqref{diffCNC} becomes the following $G_K$-equivariant isomorphism
 \begin{equation}\label{diffCNC2}
 \pi_i \TP(\sT_{\sO_\Cu};\Z_p)\otimes_{\Ainf} \widehat{\Acry} \simeq \pi_i\TP(\sT_{\sO_\Cu/p};\Z_p),
 \end{equation}
 where $\widehat{\Acry}$ is the completion of $\Acry$ with respect to the Nygaard filtration (see \cite[Definition 8.9]{BMS2}). The problem is that $\mu$ is a $0$-divisor in $\widehat{\Acry}$ (see \cite[Corollary 2.11 and Corollary 2.12]{remarkonK1}). Thus, the non-commutative analogue of \eqref{BMSmulocal} becomes the trivial equation. Therefore we will study $\pi_i\TCn(\sT/\bS[z];\
\Z_p)$ in more detail. In section \ref{sec3}, we will show that the dual Breuil-Kisin module $\pi_i\TCn(\sT/\bS[z];\Z_p)^{\vee}$ admits a Breuil-Kisin $G_K$-module structure in the sense of Gao \cite{Gao}. In section \ref{sec4}, using Du-Liu's work on $(\varphi,\hat{G})$-module \cite{DuLiu},
we will prove the following.

\begin{remark}
As pointed out by the referee, the issue that only the Nygaard-completed $A_{\mathrm{crys}}$ arises naturally from $THH$ has recently been resolved by Mao \cite[Prop.~1.8]{Mao}, who constructs a decompleted version of $\TP(\mathcal{T}_{\sO_{\Cu}/p};\Z_p)$. Although we have not reworked the present arguments in that formalism, we expect that Mao’s approach could simplify some of the proofs. 
\end{remark}

\begin{thm}\label{main}
Let $\sT$ be a smooth and proper $\sO_K$-linear category. Then there exists an integer $n\ge 0$ such that, for all $i\ge n$, the following statements hold:
\begin{itemize}


 \item[(1)] The $\Z_p[G_K]$-module $T_{\Ainf}(\pi_i \TCn(\sT/\bS[z];\Z_p)^{\vee})$ is a $\Z_p$-lattice of a crystalline representation, and there is an isomorphism of $\Bcry$-module:
\[
\pi_i\TP(\sT_k;\Z_p)^{\vee}\otimes_{W} \Bcry \simeq T_{\Ainf}(\pi_i \TCn(\sT/\bS[z];\Z_p)^{\vee})\otimes_{\Z_p}\Bcry
\]
which is compatible with $G_K$-action and Frobenius endomorphism.

 \item[(2)] If $\sT_{\Cu}$ admits a geometric realization, then there is a $G_K$-equivariant isomorphism 
 \[
 T_{\Ainf}(\pi_i \TCn(\sT/\bS[z];\Z_p)^{\vee}) \simeq \pi_i\LK K(\sT_\Cu)^{\vee}
 \]
 of $\Z_p$-modules.
\end{itemize}
\end{thm}

We prove the following as a corollary.

\begin{thm}[Main Theorem, Theorem~\ref{mainmainmain1}]\label{maininintro}
    Let $\sT$ be a smooth proper $\sO_K$-linear category. If $\sT_{\Cu}$ admits a geometric realization, then Conjecture~\ref{NCcryconj} holds for $\sT$.
\end{thm}
\begin{remark}
    In the case $\mathcal{T}$ embeds fully faithfully into $\perf(X)$ with $X$ smooth proper over $\mathcal{O}_K$, Theorem~1.10 can be deduced from the crystalline comparison for smooth proper schemes via a Fourier-Mukai kernel argument. 
\end{remark}


\section{Non-commutative version of Breuil-Kisin cohomology}


\subsection{Breuil-Kisin modules and Breuil-Kisin cohomology theory $R\Gamma_{\mS}$} Let us start by recalling the theory of Breuil–Kisin modules.
\begin{defn}
  A Breuil–Kisin module is a finitely generated $\mS$-module $M$ equipped with a  $\mS$-linear isomorphism
  \[
  \varphi_M:M\otimes_{\mS,\varphi} \mS[\frac{1}{E}] \simeq M[\frac{1}{E}].
  \]
\end{defn}

For a Breuil-Kisin module $\fM$, let us denote $\fM^{\vee}=\Hom_{\mS}(\fM,\mS)$. We note $\fM^{\vee}=\Hom_{\mS}(\fM,\mS)$ is a Breuil-Kisin module whose Frobenius map $\varphi_{\fM^{\vee}}$ is given by 
\begin{equation}\label{dualFrob}
  \fM^\vee \otimes_{\mS,\varphi} \mS[\frac{1}{E}] \simeq \Hom_{\mS[\frac{1}{E}]}(\fM\otimes_{\mS,\varphi} \mS[\frac{1}{E}],\mS[\frac{1}{E}]) \overset{\varphi_{\fM}^\vee}{\simeq} \Hom_{\mS[\frac{1}{E}]}(\fM[\frac{1}{E}],\mS[\frac{1}{E}]) \simeq \fM^\vee[\frac{1}{E}],
\end{equation}
where we use facts that $\varphi:\mS\to\mS$ and $\mS\to \mS[\frac{1}{E}]$ are flat. We note that $\fM^{\vee}$ is finite free $\mS$-module, this follows from the facts that $\text{gl.dim}~\mS=2$.

\begin{lemma}[{\cite[Corollaire 11.1.14]{FarguesFontaine}} and {\cite[Lemma 4.27]{BMS1}}]\label{BKlemma} Let $\ol{\phi}:\mS \to \Ainf$ be the map that sends $z$ to $[\pi^\flat]^p$ and is Frobenius on $W$. Let $M$ be a Breuil–Kisin module, and let $M_{\Ainf}=M\otimes_{\mS,\ol{\phi}}\Ainf$. Then $M'[\frac{1}{p}]=M_{\Ainf}[\frac{1}{p}]\otimes_{\Ainf[\frac{1}{p}]}K_0$ is a finite free $K_0$-module equipped with a Frobenius automorphism. Fix a section $k \to \sO_{\Cu}/p$, then there is a (noncanonical) $\varphi$-equivariant isomorphism
\[
M_{\Ainf}\otimes_{\Ainf}\Bcry \simeq M'[\frac{1}{p}]\otimes_{K_0}\Bcry.
\]
\end{lemma}
In \cite{BMS2}, Bhatt-Morrow-Scholze constructed a cohomology theory valued in Breuil–Kisin modules for smooth proper formal schemes over $\sO_K$. Let $\varphi_{\Ainf}:\Ainf \to \Ainf$ be the Frobenius endomorphism of $\Ainf$. Let $\phi:\mS \to \Ainf$ be the $W$-linear map that sends $z$ to $[\pi^\flat]$. Note that the diagram 
\begin{equation}\label{SAinfdiag}
\xymatrix
{\mS \ar[r]^{\varphi} \ar[d]_{\phi} \ar[rd]^{\ol{\phi}}& \mS \ar[d]^{\phi} \\
\Ainf \ar[r]_{\varphi_{\Ainf}}^{\simeq } & \Ainf
}
\end{equation}
commutes.
\begin{thm}[{\cite{BMS2}}]\label{BKcoh}
 Let $X/\sO_K$ be a smooth proper formal scheme. Then there exists a $\mS$-linear cohomology theory $R\Gamma_{\mS}(X)$ equipped with a $\varphi$-semilinear map, with the following properties:
 \begin{itemize}
   \item[(1)] All $H^i_\mS(X) :=H^i(R\Gamma_\mS(X))$ are Breuil-Kisin modules.
   \item[(2)] (\'etale comparison) After scalar extension along $\ol{\phi}:\mS\to \Ainf$, one recovers \'etale cohomology of the generic fiber
   \[
   R\Gamma_{\mS}(X) \otimes_{\mS} \Ainfmu^{\wedge}_p \simeq R\Gamma_{\et}(X_\Cu,\Z_p)\otimes_{\Z_p} \Ainfmu^{\wedge}_p
   \]
   \item[(3)] (crystalline comparison) After scalar extension along the map $\mS\to W$, which is the Frobenius on $W$ and sends $z$ to $0$, one recovers crystalline cohomology of the special fibre:
   \[
   R\Gamma_\mS(X)\otimes_{\mS}^{\mathbb{L}}W \simeq R\Gamma_{\cry}(X_k/W).
   \]
   \item[(4)] (de Rham comparison) After scalar extension along the map $\tilde{\theta}\circ \varphi:\mS \to \sO_K$, one recovers de Rham cohomology:
   \[
   R\Gamma_{\mS}(X)\otimes_{\mS}^{\mathbb{L}}\sO_K \simeq R\Gamma_{\text{dR}}(X/\sO_K).
   \]
   \end{itemize}
\end{thm}

\subsection{Perfect modules and K\"unneth formula} Let $(\mathcal{A},\otimes,1_{\mathcal{A}})$ be a symmetric monoidal, stable $\infty$-category with biexact tensor product. Firstly, we recall from \cite[section 1]{KunnethTP}.
\begin{defn}[{\cite[Definition 1.2]{KunnethTP}}]
  An object $X\in A$ is perfect if it belongs to the thick subcategory generated by the unit.
\end{defn}
For a lax symmetric monoidal, exact $\infty$-functor $F:\mathcal{A} \to \Sp$, $F(1_{\mathcal{A}})$ is naturally an $\mathbb{E}_\infty$-ring. For any $X,Y\in \mathcal{A}$, we have a natural map
\begin{equation}\label{laxmon}
F(X)\otimes_{F(1_{\mathcal{A}})} F(Y) \to F(X\otimes Y).
\end{equation}
Since $F$ is exact, if $X$ is perfect, then the map \eqref{laxmon} is an equivalence, and $F(X)$ is a perfect $F(1_{\mathcal{A}})$-module. 

We regard $\sO_K$ as an $\bS [z]$-algebra via $z\mapsto \pi$. There is a symmetric monoidal $\infty$-functor
\[
\THH(-/\bS[z];\Z_p):\Cat (\sO_K) \to \Mod_{\THH(\sO_K/\bS[z];\Z_p)}(\Sp^{BS^1}).
\]
Let us study this functor from \cite[section 11]{BMS2}. By the base change along $\bS[z] \to \bS[z^{1/{p^\infty}}]~|~z\mapsto z$, there is a natural equivalence
\begin{equation}\label{basechange}
\THH(\sO_K/\bS[z]) \otimes_{\bS[z]} \bS[z^{1/{p^\infty}}] \simeq \THH(\sO_K[\pi^{1/p^{\infty}}]/\bS[z^{1/{p^\infty}}]).
\end{equation}
The natural map $\THH(\bS[z^{1/{p^{\infty}}}];\Z_p) \to \bS[z^{1/{p^{\infty}}}]^{\wedge}_p$ is an equivalence (see \cite[Proposition. 11.7]{BMS2}), we obtain equivalences
\begin{eqnarray*}
\THH(\sO_K[\pi^{1/p^{\infty}}];\Z_p) &\simeq & \THH(\sO_K[\pi^{1/p^{\infty}}]/\bS[z^{1/{p^\infty}}];\Z_p)  \\
&\overset{\eqref{basechange}}{\simeq}&\THH(\sO_K/\bS[z];\Z_p) \otimes_{\bS[z]^{\wedge}_p} \bS[z^{1/{p^\infty}}]^{\wedge}_p.
\end{eqnarray*}
Since a morphism of homotopy groups $\pi_*(\bS[z]^{\wedge}_p) \to \pi_*(\bS[z^{1/{p^\infty}}]^{\wedge}_p)$ is faithfully flat and there is an isomorphism
$\pi_*\THH(\sO_K[\pi^{1/p^{\infty}}];\Z_p) \simeq \sO_K[\pi^{1/p^{\infty}}][u]$, where $\deg u=2$ (see \cite[section. 6]{BMS2}), we obtain an isomorphism
\begin{equation}\label{BMScalcu}
\pi_*\THH(\sO_K/\bS[z];\Z_p) \simeq \sO_K[u],
\end{equation}
where $u$ has degree $2$ (see \cite[Proposition. 11.10]{BMS2}).
\begin{prop}\label{propdualisperf}
Any dualizable object in $\Mod_{\THH(\sO_K/\bS[z];\Z_p)}(\Sp^{BS^1})$ is perfect.
\end{prop}
\begin{proof}
By the isomorphism \eqref{BMScalcu}, $\pi_*\THH(\sO_K/\bS[z];\Z_p)$ is a regular noetherian ring of finite Krull dimension concentrated in even degrees. We obtain the claim by \cite[Theorem 2.15]{KunnethTP}.
\end{proof}
\begin{prop}\label{TPTCkunneth}
Let $\sT_1,\sT_2$ be $\sO_K$-linear categories and suppose $\sT_1$ is smooth and proper. Then $\THH(\sT_1/\bS[z];\Z_p)$ is perfect in $\Mod_{\THH(\sO_K/\bS[z];\Z_p)}(\Sp^{BS^1})$, and $\TCn(\sT_1/\bS[z];\Z_p)$ (resp. $\TP(\sO_K/\bS[z];\Z_p)$) is a perfect $\TCn(\sT_1/\bS[z];\Z_p)$-module (resp. $\TP(\sO_K/\bS[z];\Z_p)$-module) and the natural map
\[
\TCn(\sT_1/\bS[z];\Z_p) \otimes_{\TCn(\sO_K/\bS[z];\Z_p)}\TCn(\sT_2/\bS[z];\Z_p) \to \TCn(\sT_1\otimes_{\sO_K}\sT_2/\bS[z];\Z_p)
\]
\[
(\text{resp. }\TP(\sT_1/\bS[z];\Z_p) \otimes_{\TP(\sO_K/\bS[z];\Z_p)}\TP(\sT_2/\bS[z];\Z_p) \to \TP(\sT_1\otimes_{\sO_K}\sT_2/\bS[z];\Z_p) \quad )
\]
is an equivalence.
\end{prop}
\begin{proof}
Note that $\infty$-functors 
\[
(-)^{hS^1}:\Mod_{\THH(\sO_K/\bS[z];\Z_p)}(\Sp^{BS^1}) \to \Mod_{\TCn(\sO_K/\bS[z];\Z_p)}(\Sp)
\]
and \[
(-)^{tS^1}:\Mod_{\THH(\sO_K/\bS[z];\Z_p)}(\Sp^{BS^1}) \to \Mod_{\TCn(\sO_K/\bS[z];\Z_p)}(\Sp)\] are lax symmetric monoidal, exact (see \cite[Corollary I.4.3]{NS18}). It is enough to show $\THH(\sT_1/\bS[z];\Z_p)$ is perfect in $\Mod_{\THH(\sO_K/\bS[z];\Z_p)}(\Sp^{BS^1})$. Since a $\infty$-functor
\[
\THH(-/\bS[z];\Z_p): \Cat(\sO_K) \to\Mod_{\THH(\sO_K/\bS[z];\Z_p)}(\Sp^{BS^1})
\]
is symmetric monoidal, and $\sT_1$ is dualizable in $\Cat(\sO_K)$ (Cf. \cite[Ch. 11]{Lurie17}). Thus $\THH(\sT_1/\bS[z];\Z_p)$ is dualizable in $\Mod_{\THH(\sO_K/\bS[z];\Z_p)}(\Sp^{BS^1})$. By Proposition \ref{propdualisperf}, we obtain the claim. 
\end{proof}

\subsection{Breuil-Kisin module of stable $\infty$-categories} Let us recall Antieau-Mathew-Nikolaus's comparison theorem of symmetric monoidal $\infty$-functors from \cite{KunnethTP}.
\begin{prop}[{\cite[Proposition 4.6]{KunnethTP}}]\label{compmonoid}
Let $\sT,\hat{\sT}$ be symmetric monoidal $\infty$-categories. Let $F_1,F_2:\sT\to\hat{\sT}$ be symmetric monoidal functors and let $t:F_1\Longrightarrow F_2$ be a symmetric monoidal natural transformation. Suppose every object of $\sT$ is dualizable. Then $t$ is an equivalence.
\end{prop}
In \cite[Proposition 11.10]{BMS2}, Bhatt-Morrow-Scholze showed the followings: on homotopy groups, there exists isomorphisms
\[
\pi_* \TCn(\sO_K/\bS[z];\Z_p) \simeq \mS[u,v]/(uv-E)
\]
where $u$ is of degree $2$ and $v$ is of degree $-2$, and 
\[
\pi_* \TP(\sO_K/\bS[z];\Z_p) \simeq \mS[\sigma^{\pm}]
\]
where $\sigma$ is of degree $2$. Let $\varphi$ be the endomorphism of $\mS$ determined by the Frobenius on $W$ and $z\mapsto z^{p}$. Nikolaus-Scholze \cite{NS18} construct two maps
\[
\can,\varphi^{hS^1}_\sT:\TCn(\sT/\bS[z];\Z_p) \rightrightarrows \TP(\sT/\bS[z];\Z_p).
\]
In \cite[Proposition 11.10]{BMS2}, Bhatt-Morrow-Scholze also showed the morphism 
\[
\can_{\sO_K}:\pi_*\TCn(\sO_K/\bS[z];\Z_p) \to \pi_*\TP(\sO_K/\bS[z];\Z_p),
\]
is $\mS$-linear and sends $u$ to $E\sigma$ and $v$ to $\sigma^{-1}$, and the morphism
\[
\varphi^{hS^1}_{\sO_K}:\pi_*\TCn(\sO_K/\bS[z];\Z_p) \to \pi_*\TP(\sO_K/\bS[z];\Z_p)
\]
is $\varphi$-linear and sends $u$ to $\sigma$ and $v$ to $\varphi(E) \sigma^{-1}$. Since $\sigma$ is an invertible element in $\TP(\sO_K/\bS[z];\Z_p)$, $\varphi^{hS^1}_{\sO_K}$ induces a map
\begin{equation}\label{localFrob}
\tilde{\varphi}^{hS^1}_{\sT}:\TCn(\sT/\bS[z];\Z_p)[\frac{1}{u}] \to \TP(\sT/\bS[z];\Z_p)
\end{equation}
for an $\sO_K$-linear stable $\infty$-category $\sT$. We note that on homotopy groups, the morphism
\[  \tilde{\varphi}^{hS^1}_{\sO_K}:\pi_*\TCn(\sO_K/\bS[z];\Z_p)[\frac{1}{u}] \to \pi_*\TP(\sO_K/\bS[z];\Z_p)
\]
is given by $\mS[u^{\pm}] \to \mS[\sigma^{\pm}]$ which is $\varphi$-semi-linear and sends $u$ to $\sigma$.
\begin{lemma}\label{periodlemma}
Let $\sT$ be a smooth proper $\sO_K$-linear category. Then the followings hold:
\begin{itemize}
  \item[(1)] $\pi_*\TCn(\sT/\bS[z];\Z_p)$ is a finitely generated $\pi_*\TCn(\sO_K/\bS[z];\Z_p)$-module, thus $\pi_i\TCn(\sT/\bS[z];\Z_p)$ is a finitely generated $\mS$-module for all $i$.
  \item[(2)] There is a natural number $n$ satisfying that for any $j \geq n$, $\pi_j\TCn(\sT/\bS[z];\Z_p)\overset{u\cdot }{\to} \pi_{j+2}\TCn(\sT/\bS[z];\Z_p)$ is an isomorphism. 
\end{itemize}
\end{lemma}
\begin{proof}
Claim (1) directly follows from Proposition ~\ref{TPTCkunneth}. For any $j \geq 0$, by the calculation of $\pi_* \TCn(\sO_K/\bS[z];\Z_p)$ (see \cite[Proposition 11.10]{BMS2}), we know $\pi_j\TCn(\sO_K/\bS[z];\Z_p)\overset{u\cdot}{\to} \pi_{j+2}\TCn(\sO_K/\bS[z];\Z_p)$ is an isomorphism. This yields claim (2) by claim (1).
\end{proof}
\begin{thm}\label{BKtheorem}
Let $\sT$ be a smooth proper $\sO_K$-linear category. Then there is a natural number $n$ such that the homotopy group $\pi_i \TCn(\sT/\bS[z];\Z_p)$ has a natural structure of a Breuil-Kisin module for any $i\geq n$, and the dual $\pi_i \TCn(\sT/\bS[z];\Z_p)^{\vee}$ also has a natural structure of a Breuil-Kisin module of finite $E$-height for any $i\geq n$. 
\end{thm}
\begin{proof}
The morphism $\tilde{\varphi}^{hS^1}_\sT$ induces a morphism of $\TP(\sO_K/\bS[z];\Z_p)$-modules:
\begin{equation}\label{TCTP}
\TCn(\sT/\bS[z];\Z_p) [\frac{1}{u}]\otimes_{\TCn(\sO_K/\bS[z];\Z_p)[\frac{1}{u}],\tilde{\varphi}^{hS^1}_{\sO_K}} \TP(\sO_K/\bS[z];\Z_p) \to \TP(\sT/\bS[z];\Z_p).
\end{equation}
By proposition~\ref{TPTCkunneth}, both sides of the map \eqref{TCTP} yields symmetric monoidal functors from $\Catsat(\sO_K)$ to $\Mod_{\TP(\sO_K/\bS[z];\Z_p)}(\Sp)$, and the map \eqref{TCTP} yields a symmetric monoidal natural transformation between them. By Proposition~\ref{compmonoid}, the morphism \eqref{TCTP} is an equivalence. On $0$-th homotopy group, $\tilde{\varphi}^{hS^1}_{\sO_K}$ is given by $\varphi:\mS\to \mS$. Since $\varphi$ is flat, one has an isomorphism
\begin{equation}\label{preTCnKisin}
\pi_i\TCn(\sT/\bS[z];\Z_p)[\frac{1}{u}] \otimes_{\mS,\varphi} \mS \simeq \pi_i \TP(\sT/\bS[z];\Z_p)
\end{equation}
for any $i$. By Lemma~\ref{periodlemma} (2), one obtains an isomorphism 
\[
\pi_i\TCn(\sT/\bS[z];\Z_p) \simeq \pi_i\TCn(\sT/\bS[z];\Z_p)[\frac{1}{u}] 
\]
for any $i \geq n$, thus we have a $\mS$-linear isomorphism on homotopy groups:
\begin{equation}\label{TCnKisin}
\pi_i\TCn(\sT/\bS[z];\Z_p) \otimes_{\mS,\varphi} \mS \simeq \pi_i \TP(\sT/\bS[z];\Z_p)
\end{equation}
for any $i \geq n$.

After inverting $E\in \mS \simeq \pi_0 \TCn(\sO_K/\bS[z];\Z_p)$, the morphism $\can_\sT$ induces a morphism of $\TP(\sO_K/\bS[z];\Z_p)[\frac{1}{E}]$-module:
\begin{equation}\label{TCTPcan}
\TCn(\sT/\bS[z];\Z_p)[\frac{1}{E}] \otimes_{\TCn(\sO_K/\bS[z];\Z_p)[\frac{1}{E}],\can_{\sO_K}} \TP(\sO_K/\bS[z];\Z_p)[\frac{1}{E}] \to \TP(\sT/\bS[z];\Z_p)[\frac{1}{E}].
\end{equation}
Both sides of the map \eqref{TCTPcan} yield symmetric monoidal functors from $\Catsat(\sO_K)$ to $\Mod_{\TP(\sO_K/\bS[z];\Z_p)[\frac{1}{E}]}(\Sp)$, and the map \eqref{TCTPcan} yield a symmetric monoidal natural transformation between them. Thus the morphism \eqref{TCTPcan} is an equivalence. Note that the morphism 
\[
\can_{\sO_K}[\frac{1}{E}]:\pi_*\TCn(\sO_K/\bS[z];\Z_p)[\frac{1}{E}] \to \pi_*\TP(\sO_K/\bS[z];\Z_p)[\frac{1}{E}]
\]
is an isomorphism (see \cite[Proposition 11.10]{BMS2}). This yields an isomorphism
\begin{equation}\label{canisom}
\can_{\sT}[\frac{1}{E}]:\pi_*\TCn(\sT/\bS[z];\Z_p)[\frac{1}{E}] \simeq \pi_*\TP(\sT/\bS[z];\Z_p)[\frac{1}{E}].
\end{equation}
Combine \eqref{canisom} with \eqref{TCnKisin}, we obtain an isomorphism
\begin{equation}\label{structuremapTCn}
\pi_i\TCn(\sT/\bS[z];\Z_p) \otimes_{\mS,\varphi} \mS [\frac{1}{E}]\overset{\eqref{TCnKisin}}{\simeq} \pi_i \TP(\sT/\bS[z];\Z_p)[\frac{1}{E}] \overset{\eqref{canisom}}{\simeq} \pi_i\TCn(\sT/\bS[z];\Z_p)[\frac{1}{E}] 
\end{equation}
for any $i \geq n$. Let us study the dual of \eqref{structuremapTCn}. We have a morphism
\begin{align}\label{dualstruc}
\pi_i\TCn(\sT/\bS[z];\Z_p)^{\vee} \otimes_{\mS,\varphi}\mS &= \Hom_{\mS}(\pi_i\TCn(\sT/\bS[z];\Z_p) \otimes_{\mS,\varphi}\mS, \mS) \\
    \overset{\eqref{TCnKisin}^{\vee}}{\simeq} \Hom_{\mS}(\pi_i\TP(\sT/\bS[z];\Z_p), \mS) &\overset{\text{can}_{\sT}^{\vee}}{\to} \Hom_{\mS}(\pi_i\TCn(\sT/\bS[z];\Z_p), \mS) \nonumber
\end{align}
After localization by $E$, the morphism  \eqref{dualstruc} coincides with the dual of \eqref{structuremapTCn}. 
\end{proof} 

\subsection{The comparison between $\TCn(\sT/\bS[z];\Z_p)$ and $\TP(\sT_{\sO_\Cu};\Z_p)$} In this section, for a smooth proper $\sO_K$-linear category $\sT$, we will compare $\TCn(\sT/\bS[z];\Z_p)$ with $\TP(\sT_{\sO_\Cu};\Z_p)$. Denote $K_\infty=K(\pi^{1/p^{\infty}})$ and by $\sO_{K_\infty}$ by the integer ring of $K_\infty$.
\begin{lemma}[{\cite[Corollary 11.8]{BMS2}}]\label{BMSTHHlemma}
 For any $\sO_K$-linear stable $\infty$-category $\mathcal{T}$, the natural map
 \[ \THH(\mathcal{T}_{\sO_\Cu};\Z_p) 
 \to \THH(\mathcal{T}_{\sO_\Cu}/\bS[z^{1/p^{\infty}}];\Z_p) 
 \]
 is an equivalence which is compatible with $S^1$-action, and $G_{K_\infty}$-action. In particular, the natural map
 \[ \TP(\mathcal{T}_{\sO_\Cu};\Z_p) 
 \to \TP(\mathcal{T}_{\sO_\Cu}/\bS[z^{1/p^{\infty}}];\Z_p)
 \]
 is an equivalence which is compatible with $G_{K_\infty}$-action.
\end{lemma}
\begin{proof}
The morphism $\bS[z^{1/p^{\infty}}] \to \sO_{K_\infty}~|~ z^{1/p^n}\mapsto \pi^{1/p^n}$ fits into the following commutative diagram
\begin{equation}\label{diagram}
\xymatrix{
 &\sO_K\ar[r] &\sO_{K_\infty} \ar[r] & \sO_{\Cu}\\
\bS\ar[r] &\bS[z]\ar[r] \ar[u]& \bS[z^{1/p^{\infty}}]\ar[u] &
}
\end{equation}
The diagram yields a map 
\[
\THH(\mathcal{T}_{\sO_\Cu};\Z_p) 
 \to \THH(\mathcal{T}_{\sO_\Cu}/\bS[z^{1/p^{\infty}}];\Z_p)\simeq \THH(\mathcal{T}_{\sO_\Cu};\Z_p)\otimes_{\THH(\bS[z^{1/p^{\infty}}];\Z_p)}\bS[z^{1/p^{\infty}}]^{\wedge}_{p}
\]
which is compatible with $S^1$-action. By \cite[Proposition 11.7]{BMS2}, the natural map $\THH(\bS[z^{1/p^{\infty}}];\Z_p) \to \bS[z^{1/p^{\infty}}]^{\wedge}_{p}$ is an equivalence. Since $\pi^{1/p^n}$ is in $K_\infty$ for any $n$, thus the equivalence is $G_{K_\infty}$-equivariant.
\end{proof}

\begin{prop}\label{THHOCperfect}
  For a smooth proper $\sO_K$-linear category $\sT$, $\THH(\sT_{\sO_\Cu}/\bS[z^{1/p^{\infty}}];\Z_p)$ is a perfect object in $\Mod_{\THH({\sO_\Cu}/\bS[z^{1/p^{\infty}}];\Z_p)}(\Sp^{BS^1})$. In particular, $\TP(\sT_{\sO_\Cu}/\bS[z^{1/p^{\infty}}];\Z_p)$ is perfect in $\Mod_{\TP({\sO_\Cu}/\bS[z^{1/p^{\infty}}];\Z_p)}(\Sp)$.
\end{prop}
\begin{proof}
  We already know $\THH(\sT/\bS[z];\Z_p)$ is a perfect object in $\Mod_{\THH(\sO_K/\bS[z];\Z_p)}(\Sp^{BS^1})$ (see Proposition~\ref{TPTCkunneth}). Besides, the functor 
  \[
  -\otimes_{\THH(\sO_K/\bS[z];\Z_p)}\THH(\sO_{\Cu}/\bS[z];\Z_p) :\Mod_{\THH(\sO_K/\bS[z];\Z_p)}(\Sp^{BS^1}) \to \Mod_{\THH(\sO_{\Cu}/\bS[z];\Z_p)}(\Sp^{BS^1})
  \]
  is exact and sends $\THH(\sT/\bS[z];\Z_p)$ to $\THH(\sT_{\sO_\Cu}/\bS[z];\Z_p)$, $\THH(\sT_{\sO_\Cu}/\bS[z];\Z_p)$ is a perfect object in $\Mod_{\THH(\sO_{\Cu}/\bS[z];\Z_p)}(\Sp^{BS^1})$. Since the functor 
  \[
  -\otimes_{\THH(\bS[z^{1/p^{\infty}}]/\bS[z];\Z_p)}\bS[z^{1/p^{\infty}}]^{\wedge}_p :\Mod_{\THH(\sO_{\Cu}/\bS[z];\Z_p)}(\Sp^{BS^1}) \to \Mod_{\THH(\sO_{\Cu}/\bS[z^{1/p^\infty}];\Z_p)}(\Sp^{BS^1})
  \]
  is exact and sends $\THH(\sT_{\sO_\Cu}/\bS[z];\Z_p)$ to $\THH(\sT_{\sO_\Cu}/\bS[z^{1/p^{\infty}}];\Z_p)$, thus $\THH(\sT_{\sO_\Cu}/\bS[z^{1/p^{\infty}}];\Z_p)$ is a perfect object in $\Mod_{\THH(\sO_{\Cu}/\bS[z^{1/p^\infty}];\Z_p)}(\Sp^{BS^1})$.
\end{proof}

Lemma~\ref{BMSTHHlemma} and Proposition~\ref{THHOCperfect} imply the following.
\begin{prop}\label{THHOCbSperf}
  For a smooth proper $\sO_K$-linear category $\sT$, $\THH(\sT_{\sO_\Cu};\Z_p)$ is a perfect object in $\Mod_{\THH({\sO_\Cu};\Z_p)}(\Sp^{BS^1})$.
\end{prop} 

The following is a non-commutative version of the comparison theorem between $R\Gamma_\mS$ and $R\Gamma_{\Ainf}$ in \cite[Theorem 1.2 (1)]{BMS2}
\begin{thm}\label{TPTCcomp}
Let $\sT$ be a smooth proper $\sO_K$-linear category. Then there is a natural number $n$ satisfying that there is a $G_{K_\infty}$-equivariant isomorphism
\[
\pi_i \TCn(\sT/\bS[z];\Z_p)\otimes_{\mS,\ol{\phi}} \Ainf \simeq \pi_i \TP(\sT_{\sO_\Cu}/\bS[z^{1/p^{\infty}}];\Z_p)\simeq \pi_i \TP(\sT_{\sO_\Cu};\Z_p)
\]
for any $i \geq n$, where $\ol{\phi}$ is the map which sends $z$ to $[\pi^{\flat}]^p$ and is the Frobenius on $W$, and $g\in G_{K_\infty}$ acts on $1\otimes g$ on left hand side.
\end{thm}
\begin{proof}
For a smooth proper $\sO_K$-linear category $\sT$, consider the following morphism
\[
\TCn(\sT/\bS[z];\Z_p) \overset{\varphi^{hS^1}_{\sT}}{\to} \TP(\sT/\bS[z];\Z_p) \to \TP(\sT_{\sO_\Cu}/\bS[z^{1/p^{\infty}}];\Z_p)
\]
where the second map is given by the diagram~\eqref{diagram}. The map sends $u$ to $\sigma$ and $u$ is an invertible element in $\TP(\sT_{\sO_\Cu}/\bS[z^{1/p^{\infty}}];\Z_p)$, we have a morphism
\begin{equation}
  \underline{\varphi}^{hS^1}_{\sT}:\TCn(\sT/\bS[z];\Z_p)[\frac{1}{u}] \to \TP(\sT_{\sO_\Cu}/\bS[z^{1/p^{\infty}}];\Z_p).
\end{equation}
Let us prove that the map $\underline{\varphi}^{hS^1}_{\sT}$ yields an equivalence of $\TP(\sT_{\sO_\Cu}/\bS[z^{1/p^{\infty}}];\Z_p)$-module:
\begin{equation}\label{AKAMO}
\TCn(\sT/\bS[z];\Z_p)[\frac{1}{u}] \otimes_{\TCn(\sO_K/\bS[z];\Z_p)[\frac{1}{u}],\underline{\varphi}^{hS^1}_{\sO_K}}\TP(\sO_{\Cu}/\bS[z^{1/p^{\infty}}];\Z_p)\simeq\TP(\sT_{\sO_\Cu}/\bS[z^{1/p^{\infty}}];\Z_p)
\end{equation}
which is compatible with $G_{K_\infty}$-action. By Proposition~\ref{TPTCkunneth}, the left side of the map \eqref{AKAMO} yields a symmetric monoidal functor from $\Catsat(\sO_K)$ to $\Mod_{\TP(\sO_{\Cu}/\bS[z^{1/p^{\infty}}];\Z_p)}(\Sp)$. By Proposition~\ref{THHOCperfect}, the right side of the map \eqref{AKAMO} also yields a symmetric monoidal functor from $\Catsat(\sO_K)$ to $\Mod_{\TP(\sO_{\Cu}/\bS[z^{1/p^{\infty}}];\Z_p)}(\Sp)$. By \cite[section 11]{BMS2}, on homotopy groups, the morphism
\[
\underline{\varphi}^{hS^1}_{\sO_K}:\pi_*\TCn(\sO_K/\bS[z];\Z_p)[\frac{1}{u}] \to \pi_*\TP(\sO_{\Cu}/\bS[z^{1/p^\infty}];\Z_p)
\]
is given by $\mS[u^{\pm}] \to \Ainf[\sigma^{\pm}]$ which is $\ol{\phi}$-linear and sends $u$ to $\sigma$, thus it is flat by \cite[Lemma 4.30]{BMS1}. We now have an $G_{K_\infty}$-equivariant isomorphism
\[
\pi_i\TCn(\sT/\bS[z];\Z_p)[\frac{1}{u}] \otimes_{\mS,\ol{\phi}} \Ainf\simeq \pi_i\TP(\sT_{\sO_\Cu}/\bS[z^{1/p^\infty}];\Z_p).
\]
for any $i$. By Lemma \ref{periodlemma} (2), we obtain the claim.
\end{proof}

\subsection{The comparison between $\TCn(\sT/\bS[z];\Z_p)$ and $\TP(\sT_{k};\Z_p)$} In this section, for a smooth proper $\sO_K$-linear category $\sT$, we will compare $\TCn(\sT/\bS[z];\Z_p)$ with $\TP(\sT_{k};\Z_p)$. There exists a Cartesian diagram of $\mathbb{E}_\infty$-ring:
\[
\xymatrix{
k & \sO_K \ar[l] \\
\bS \ar[u] & \bS[z] \ar[u] \ar[l]
}
\]
For an $\sO_K$-linear stable $\infty$-category $\sT$, the diagram yields morphisms
\[
\TCn(\sT/\bS[z];\Z_p) \to \TCn(\sT_k;\Z_p)
\]
and 
\begin{equation}\label{TCnTOkcomp}
\TCn(\sT/\bS[z];\Z_p) \to \TCn(\sT_k;\Z_p) \overset{\varphi^{hS^1}_k}{\to} \TP(\sT_k;\Z_p).
\end{equation}
\begin{thm}\label{TCnTPcomp3}
Let $\sT$ be a smooth proper $\sO_K$-linear category. Then there is a natural number $n$ satisfying that the morphism \eqref{TCnTOkcomp} an isomorphism
\[
\pi_i \TCn(\sT/\bS[z];\Z_p)[\frac{1}{p}]\otimes_{\mS [\frac{1}{p}],\tilde{\phi}} K_0 \to \pi_i \TP(\sT_k;\Z_p)[\frac{1}{p}]
\]
for any $i \geq n$, where $\tilde{\phi}$ is the map which sends $z$ to $0$ and Frobenius on $W$.
\end{thm}
\begin{proof}
  The morphism \eqref{TCnTOkcomp} induces a commutative diagram
  \[
  \xymatrix{
  \TCn(\sT/\bS[z];\Z_p) \ar[d]_{\varphi^{hS^1}_{\sT}} \ar[r]& \TCn(\sT_k;\Z_p) \ar[d]^{\varphi^{hS^1}_k} \\
    \TP(\sT/\bS[z];\Z_p)   \ar[r]&\TP(\sT_k;\Z_p) 
  }
  \]
Firstly, we prove the morphism $\TP(\sT/\bS[z];\Z_p) \to\TP(\sT_k;\Z_p) $ induces an isomorphism
  \[
  \pi_i \TP(\sT/\bS[z];\Z_p)[\frac{1}{p}]\otimes_{\mS [\frac{1}{p}],\gamma} K_0 \to \pi_i \TP(\sT_k;\Z_p)[\frac{1}{p}]
  \]
  for any $i$, where $\gamma$ is a $W$-linear and sends $z$ to $0$. By Theorem~\ref{BKtheorem} and \cite[Proposition 4.3]{BMS1}, there is $n$ such that for any $i\geq n$, the homotopy group $\pi_i\TCn(\sT/\bS[z];\Z_p)$ is a Breuil-Kisin module, and $\pi_i\TCn(\sT/\bS[z];\Z_p)[\frac{1}{p}]$ is a finite free $\mS[\frac{1}{p}]$-module. By an isomorphism~\ref{TCnKisin} and the fact that $\TP(\sT/\bS[z];\Z_p)$ is $2$-periodic, we obtain that $\pi_i\TP(\sT/\bS[z];\Z_p)[\frac{1}{p}]$ is a finite free $\mS[\frac{1}{p}]$-module for any $i$. Thus, on homotopy groups, one obtains an isomorphism 
  \begin{align}\label{flatmap}
 \pi_*\TP(\sT/\bS[z];\Z_p)[\frac{1}{p}]\simeq \pi_0\TP(\sT/\bS[z];\Z_p)[\frac{1}{p}]\otimes_{\mS[\frac{1}{p}]} \pi_*\TP(\sO_K/\bS[z];\Z_p)[\frac{1}{p}] \\
 \bigoplus \pi_1\TP(\sT/\bS[z];\Z_p)[\frac{1}{p}]\otimes_{\mS[\frac{1}{p}]}\pi_*\TP(\sO_K/\bS[z];\Z_p)[\frac{1}{p}],\nonumber
  \end{align}
 and we see that $\pi_*\TP(\sT/\bS[z];\Z_p)[\frac{1}{p}]$ is a flat graded $\pi_*\TP(\sO_K/\bS[z];\Z_p)[\frac{1}{p}]$-module. Let us prove the morphism
  \begin{equation}\label{ACOMU}
  \TP(\sT/\bS[z];\Z_p)[\frac{1}{p}] \otimes_{\TP(\sO_K/\bS[z];\Z_p)[\frac{1}{p}]} \TP(k;\Z_p)[\frac{1}{p}] \to \TP(\sT_k;\Z_p)[\frac{1}{p}]
  \end{equation}
  is an equivalence. Both sides of the map \eqref{ACOMU} yield symmetric monoidal functors from $\Catsat(\sO_K)$ to $\Mod_{\TP(k;\Z_p)[\frac{1}{E}]}(\Sp)$, thus by Proposition~\ref{compmonoid}, we know the map \eqref{ACOMU} is an equivalence. By the isomorphism \eqref{flatmap} and the fact that $\TP(\sT_{\sO_\Cu};\Z_p)$ is $2$-periodic, on homotopy groups, we have an isomorphism
\begin{equation}\label{TPTPcomp}
\pi_i \TP(\sT/\bS[z];\Z_p)[\frac{1}{p}]\otimes_{\mS [\frac{1}{p}],\gamma} K_0 \to \pi_i \TP(\sT_k;\Z_p)[\frac{1}{p}]
\end{equation}
for any $i$. Combine the isomorphism~\eqref{TCnKisin} with an equality $\gamma \circ \varphi=\tilde{\phi}$, there is an isomorphism
\begin{eqnarray*}
\pi_i \TCn(\sT/\bS[z];\Z_p)[\frac{1}{p}]\otimes_{\mS [\frac{1}{p}],\tilde{\phi}} K_0 & \simeq & (\pi_i \TCn(\sT/\bS[z];\Z_p)[\frac{1}{p}]\otimes_{\mS [\frac{1}{p}],\phi} \mS[\frac{1}{p}]) \otimes_{\mS[\frac{1}{p}],\tau} K_0 \\
&\overset{\eqref{TCnKisin}}{\simeq} & \pi_i \TP(\sT/\bS[z];\Z_p)[\frac{1}{p}]\otimes_{\mS [\frac{1}{p}],\tau} K_0 \\
&\overset{\eqref{TPTPcomp}}{\simeq} &\pi_i \TP(\sT_k;\Z_p)[\frac{1}{p}].
\end{eqnarray*}
for any $i\geq n$.
\end{proof}

\subsection{The comparison theorem between $\TCn(\sT/\bS[z];\Z_p)$ and $\LK K(\sT_{\Cu})$} Firstly, we prove K\"unneth formula of $K(1)$-local $K$ theory for $\Cu$-linear categories that admit a geometric realization.
\begin{prop}\label{geomKunnethprop}
For $\Cu$-linear categories $\sT_1$, $\sT_2$ which admit geometric realization, the natural map
\begin{equation}\label{geomKunneth}
\LK K(\sT_1) \otimes_{\LK K(\Cu)}\LK K(\sT_2) \to \LK (\sT_1\otimes_{\Cu} \sT_2)
\end{equation}
is an equivalence.
\end{prop}
\begin{proof} At first, we prove the claim in the case that there exist smooth proper varieties $X_1$ and $X_2$ so that $\sT_1=\perf (X_1)$ and $\sT_2= \perf (X_2)$. As a preliminary reduction, we first establish the claim over $\Cu=\bC_p$, the completion of $\overline{\Q}_p$ with respect to $|\cdot|_p$. We fix an isomorphism of fields $\sigma:\sC \simeq \bC$. We denote $X_{i}\times_{\Spec \sC} \Spec \bC$ by $X_{i,\bC}$. Due to Thomason \cite{Thomason}, we have an equivalence $L_{K(1)}K(X_i)\simeq K_{\text{top}}(X_{i,\bC}^{an})^{\wedge}_p$. Due to Atiyah \cite{Atiyah}, since a CW complex which comes from a compact complex manifold is finite, we know a natural map
\[
K_{\text{top}}(X_{1,,\bC}^{an}) \otimes_{K_{\text{top}}(\text{pt})} K_{\text{top}}(X_{2,\bC}^{an}) \to K_{\text{top}}(X_{1,\bC}^{an}\times X_{2,\bC}^{an})
\]
is an equivalence. The isomorphism of fields $\sigma:\sC \simeq \bC$ induces an equivalence of ring spectra $K_{\text{top}}(\text{pt})^{\wedge}_p \simeq \LK K(\bC)$. We see there exists a commutative diagram
\[
\xymatrix{
L_{K(1)}K(X_1) \otimes_{L_{K(1)}K(\sC)} L_{K(1)}K(X_2) \ar[d]_-{\simeq} \ar[r]&L_{K(1)}K(X_1\times_{\Spec \sC} X_2)\ar[d]^{\simeq}\\
K_{\text{top}}(X_{1,,\bC}^{an}) \otimes_{K_{\text{top}}(\text{pt})} K_{\text{top}}(X_{2,\bC}^{an})\ar[r]_-{\simeq}&K_{\text{top}}(X_{1,\bC}^{an}\times X_{2,\bC}^{an})
}
\]
and we obtain the claim. We now treat the general case. We fix an embedding of fields $\bC_p \hookrightarrow \sC$. There exists a subfield $\bC_p \subset L \subset \Cu$ s.t. $L$ is finitely generated over $\bC_p$ and $X$ and $Y$ are defined over $L$. Take a variety $T$ over $\bC_p$ such that
the function field $K(T)$ is isomorphic to $L$. Then, we can take varieties $X', Y'$ over $\bC_p$ and morphisms $X' \to T, Y' \to T$ over $\bC_p$ such that the generic fibers of $X',Y'$ tensored with $\Cu$ are $X,Y$ respectively. Then, by shrinking
$T$, we may assume $X',Y'$ are proper and smooth over $T$. Choose a $\bC_p$-value point $t$ on $T$. In that case, by using proper smooth base change theorems in \'etale cohomology and Thomason's spectral sequence \cite{Thomason}, we have canonical equivalences
\begin{eqnarray*}
    \LK K(X'_t) \simeq \LK(X) \\
    \LK K(Y'_t) \simeq \LK(Y) \\
    \LK K(X'_t \times_{\bC_p} Y'_t) \simeq \LK(X\times_{\Cu}Y)
\end{eqnarray*}
The claim follows from the $\Cu=\bC_p$-case.

Secondly, we prove the claim in the case when $\sT_1=\perf(X)$ and $\sT_2=\perf(Y)$ for a smooth variety $X$ and a smooth proper variety $Y$. Choose a good compactification $(\overline{X},\Sigma_{i=1}^n D_i)$ of $X$ where each $D_i$ is irreducible. For $n\leq i$, we write $X_i= X-\bigcup_{j=1}^i D_j$. The Verdier quotient $\perf_{D_{i}\cap X_{i-1}}(X_{i-1}) \subset \perf(X_{i-1}) \to \perf(X_{i})$ induces a fibre sequence
\begin{equation}\label{verdier1}
    \LK K(Y) \to \LK K(Y) \to \LK K(Y), 
\end{equation}
and the Verdier quotient $\perf_{(D_{i}\cap X_{i-1})\times Y}(X_{i-1}\times Y) = \perf_{D_{i}\cap X_{i-1}}(X_{i-1}) \otimes \perf(Y) \subset \perf(X_{i-1})\otimes \perf(Y) \to \perf(X_{i})\otimes \perf(Y)$ induces a fibre sequence
\begin{equation}\label{verdier2}
  \LK K((D_i\cap X_{i-1})\times Y) \to \LK K(X_{i-1}\times Y) \to \LK K(X_i\times Y).
\end{equation}
There is a morphism of fibre sequence $\LK K(D_{i}\cap X_{i-1})\otimes_{\LK K(\Cu)} \eqref{verdier1} \to \eqref{verdier2}$.
Note that $D_i\cap X_{i-1}$ have a good compactification $(D_i,\Sigma_{l=1}^{i-1} D_i\cap D_l)$. By the induction on $i$, we obtain the claim. The same argument as above proves in the case that $X$ and $Y$ are smooth varieties. 

Thirdly, we prove the claim in the case that $\sT_1=\perf(X)$ and $\sT_2=\perf(Y)$ for a separated and finite type scheme $X$ over $\Cu$ and a smooth variety $Y$. Choose a sequence of closed subschemes of $X$:
\[
\emptyset=Z_0 \subset Z_1 \subset Z_2 \subset \cdots \subset Z_n=X
\]
such that $Z_i\backslash |Z_{i-1}|$ is smooth over $\Cu$ for any $i$. For any variety $Z$ over $\Cu$, since $\Cu$ is characteristic $0$, we have an equivalence $\LK K(Z) \simeq \LK HK(Z)$ (see  \cite{Weibel}), where $HK(Z)$ is the homotopy $K$ theory of $Z$. Using Quillen's localization theorem, we have fibre sequences
\begin{equation}\label{verdier3}
    \LK K(Z_{i-1})\to  \LK K(Z_{i}) \to \LK K(Z_i\backslash |Z_{i-1}|) 
\end{equation}
and 
\begin{equation}\label{verdier4}
   \LK K(Z_{i-1}\times Y) \to \LK K(Z_{i}\times Y) \to \LK K((Z_i\backslash |Z_{i-1}|)\times Y),
\end{equation}
and there is a morphism of fibre sequences $ \LK K(Y)\otimes_{\LK K(\Cu)}\eqref{verdier3} \to \eqref{verdier4}$. By the induction on $i$, we obtain the claim. The same argument as above shows the claim in the case that $X$ and $Y$ are smooth varieties. 

For a $\Cu$-dg-algebra $B$ concentrated in non-positive degrees, using the Dundas-Goodwillie-McCarthy theorem \cite{DGM}, we have a homotopy pullback square
\[
\xymatrix{
K(B) \ar[r] \ar[d] & K(H^0(B)) \ar[d] \\
\TC(B) \ar[r]& \TC(H^0(B))
}
\]
Since $\Cu$ has characteristic $0$, we have  $\LK \TC(B)=\LK \TC (H^0(B))=0$. Thus the natural map $\LK K(B) \to \LK K(H^0(B))$ is an equivalence. For affine derived schemes $X$ and $Y$ whose truncations are separated and of finite type over $\Cu$, the claim holds for $\sT_1=\perf(X)$ and $\sT_2=\perf(Y)$. For derived schemes $X$ and $Y$ whose truncations are smooth over $\Cu$, the claim follows from the case that $X$ and $Y$ are affine derived schemes by \cite[Theorem A.4]{CMNN}. For $\Cu$-linear categories $\sT_1$, $\sT_2$ which admit geometric realization $\sT_1 \hookrightarrow \perf(X)$ and $\sT_2 \hookrightarrow \perf(Y)$, the claim follows from the fact that $\LK K(\sT_1)$ (resp. $\LK K(\sT_2)$) is a retract in $\LK K(X)$ (resp. $\LK K(Y)$) see \cite[Proposition 3.4]{admissiblesub}.
\end{proof}

We let $\Catsat(\sO_K)_{\text{geom}}$ denote the $\infty$-category of smooth proper $\sO_K$-linear categories $\sT$ such that $\sT_{\Cu}$ admits a geometric realization.

Let $\sT$ be a smooth proper $\sO_K$-linear category. Using \cite[Thm 2.16]{remarkonK1}, one obtain an equivalence $\LK K(\sT_{\sO_{\Cu}}) \simeq \LK K(\sT_\Cu)$ of ring spectra. We recall some results about topological Hochschild homology theory from \cite{Lars}, \cite{remarkonK1} and \cite{BMS2}. On homotopy groups, there is an isomorphism
\begin{equation}
  \pi_*\TP(\sO_\Cu;\Z_p) \simeq \Ainf[\sigma,\sigma^{-1}]
\end{equation}
where $\sigma$ is a generator $\sigma \in \TP_2(\sO_\Cu;\Z_p)$. Let $\beta\in K_2(\Cu)$ be the Bott element. For any $\sO_\Cu$-linear stable $\infty$-category $\sD$, we have an identification $\LK K(\sD) \simeq \LK K(\sD_\Cu)$ (see \cite[Theorem~2.16]{remarkonK1}). The cyclotomic trace map $\text{tr}:K(\sO_\Cu) \to \TP(\sO_\Cu;\Z_p)$ sends $\beta$ to a $\Z_p^*$-multiple of $([\epsilon]-1)\sigma$ (see \cite{Lars} and \cite[Theorem 1.3.6]{HN}), the cyclotomic trace map induce a morphism
\[
\LK\text{tr}:\LK K(\sD_\Cu)\simeq \LK K(\sD) \to \LK \TP(\sD)\simeq \TP(\sD;\Z_p)[\frac{1}{[\epsilon]-1}]^{\wedge}_p
\]
for any $\sO_\Cu$-linear category $\sD$.

\begin{thm}\label{KTCncomp}
Let $\sT$ be a smooth proper $\sO_K$-linear category. We assume $\sT_{\Cu}$ admits a geometric realization, then the trace map induces an equivalence
\begin{equation}\label{KTCnequi}
\LK K(\sT_\Cu)\otimes_{\LK K(\Cu)} \LK \TP(\sO_\Cu;\Z_p) \to \LK \TP(\sT_{\sO_\Cu};\Z_p).
\end{equation}
In particular, there is a $G_K$-equivariant isomorphism
\[
\pi_i \LK K(\sT_\Cu) \otimes_{\Z_p} \Ainf[\frac{1}{[\epsilon]-1}]^{\wedge}_p \simeq \pi_i \TP(\sT_{\sO_\Cu};\Z_p)[\frac{1}{[\epsilon]-1}]^{\wedge}_p
\]
for any $i$.
\end{thm}
\begin{proof}
By Proposition~\ref{THHOCbSperf} and Proposition~\ref{geomKunnethprop}, both sides of the map \eqref{KTCnequi} yield symmetric monoidal functors from $\Catsat(\sO_K)_{\text{geom}}$ to $\Mod_{\LK\TP(\sO_\Cu;\Z_p)}(\Sp)$, thus the map \eqref{KTCnequi} is an equivalence. On homotopy groups, the morphism $\LK \tr :\pi_*\LK K(\Cu)\to \pi_*\LK\TP(\sO_\Cu;\Z_p)$ is given by $\Z_p[\beta^{\pm}] \to \Ainf[\frac{1}{[\epsilon]-1}]^{\wedge}_p[\sigma^{\pm}]$ which sends $\beta$ to $([\epsilon]-1)\sigma$, and is flat. Thus we obtain the claim.
\end{proof}


\subsection{Non-commutative version of Bhatt-Morrow-Scholze's comparison theorem} In this section, as a summary of the previous several sections, we prove a non-commutative version of Theorem~\ref{BKcoh}.
\begin{thm}\label{NCBMS}
  Let $\sT$ be a smooth proper $\sO_K$-linear category. Then there is $n$ satisfying the followings:
 \begin{itemize}
   \item[(1)] For any $i\geq n$, $\pi_i \TCn(\sT/\bS[z];\Z_p)$ is a Breuil-Kisin module.
   \item[(1')] For any $i\geq n$, $\pi_i \TCn(\sT/\bS[z];\Z_p)^{\vee}$ is a Breuil-Kisin module of finite $E$-height.
   \item[(2)] (K(1)-local K theory comparison) Assume $\sT_{\Cu}$ admits a geometric realizaiton. For any $i\geq n$, after scalar extension along $\ol{\phi}:\mS\to \Ainf$, one recovers K(1)-local K theory of generic fiber
   \[
   \pi_i \TCn(\sT/\bS[z];\Z_p) \otimes_{\mS,\ol{\phi}} \Ainf[\frac{1}{\mu}]^{\wedge}_p \simeq \pi_i\LK K(\sT_{\Cu})\otimes_{\Z_p} \Ainf[\frac{1}{\mu}]^{\wedge}_p
   \]
   \item[(3)] (topological periodic homology theory comparison) For any $i\geq n$, after scalar extension along the map $\tilde{\phi}:\mS\to W$ which is the Frobenius on $W$ and sends $z$ to $0$, one recovers topological periodic homology theory of the special fiber:
  \[
   \pi_i \TCn(\sT/\bS[z];\Z_p)[\frac{1}{p}] \otimes_{\mS[\frac{1}{p}],\tilde{\phi}} K_0 \simeq \pi_i \TP(\sT_k;\Z_p)[\frac{1}{p}].
   \]
   \end{itemize}
\end{thm}
\begin{proof}
  Claims (1) and (1') follow from Theorem~\ref{BKtheorem}. Claim (2) follows from Theorem~\ref{TPTCcomp} and Theorem~\ref{KTCncomp}. Claim (3) follows from Theorem~\ref{TCnTPcomp3}.
\end{proof}
\begin{cor}\label{Bcrycor}
 For a smooth proper $\sO_K$-linear category $\sT$ which admits a geometric realization and an integer $i$, there is an isomorphism of $\Bcry$-modules:
\[
\pi_i\TP(\sT_k;\Z_p)\otimes_W \Bcry \simeq \pi_i L_{K(1)}K(\sT_\Cu)\otimes_{\Z_p} \Bcry.
\]
\end{cor}
\begin{proof}
  The claim follows from Theorem~\ref{NCBMS} amd Lemma~\ref{BKlemma}.
\end{proof}
\section{Breuil-Kisin module and semi-stable representation}\label{sec3}

\subsection{Breuil-Kisin $G_K$-module} Let us recall the work of H. Gao on Breuil-Kisin $G_K$-module \cite{Gao}. Look at the following diagram: 
\[
\xymatrix{
\Ainf \ar[r]^{\varphi_{\Ainf}} & \Ainf \\
  \mS \ar[u]^{\phi} \ar[r]_{\varphi} & \mS \ar[u]^{\phi} 
}
\]
 We note that $\phi$ sends $E$ to $\xi$. Let $(\fM,\varphi_\fM)$ be a Breuil-Kisin module. We write $\widehat{\fM}$ denote to $\fM\otimes_{\mS,\phi}\Ainf$. After tensoring $-\otimes_{\mS,\phi} \Ainf$, the $\mS[\frac{1}{E}]$-linear isomorphism $\fM\otimes_{\mS,\varphi}\mS[\frac{1}{E}]\overset{\varphi_{\fM}}{\simeq}\fM[\frac{1}{E}]$ becomes the isomorphism 
\begin{equation}\label{power}
\widehat{\varphi}_{\fM}:\widehat{\fM}\otimes_{\Ainf,\varphi_{\Ainf}} \Ainf[\frac{1}{\xi}] \simeq \widehat{\fM}[\frac{1}{\xi}].
\end{equation}

\begin{defn}[{\cite[Defn 7.1.1]{Gao}}]\label{BKGKdefn}
Let $(\fM,\varphi_\fM)$ be a finite free Breuil-Kisin module of finite $E$-height,
equipped with a continuous $\Ainf$-semi-linear $G_K$-action on 
$\widehat{\fM}=\fM\otimes_{\mS,\phi}\Ainf$. 
We call $(\fM,\varphi_\fM)$ a Breuil--Kisin $G_K$-module if this $G_K$-action satisfies:
\begin{itemize}
  \item[(1)] $G_K$ commutes with $\widehat{\varphi}_\fM$.
  \item[(2)] $\fM \subset \widehat{\fM}^{G_{K_\infty}}$ via the embedding $\fM \subset \widehat{\fM}$.
  \item[(3)] $\fM/z\fM \subset \bigl(\widehat{\fM}\otimes_{\Ainf}W(\ol{k})\bigr)^{G_K}$ via the embedding $\fM/z\fM\subset \widehat{\fM}\otimes_{\Ainf}W(\ol{k})$.
\end{itemize}
\end{defn}

The $G_K$-equivariant surjective map of rings $\sO_\Cu \to \ol{k}$ induces a morphism of ring spectra
\begin{equation}\label{TPTPOColkmap}
\TP(\sT_{\sO_\Cu};\Z_p) \to \TP(\sT_{\ol{k}};\Z_p).
\end{equation}
Combine an equivalence $\THH(\sT_{\ol{k}};\Z_p)\simeq \THH(\ol{k};\Z_p) \otimes_{\THH(\sO_\Cu;\Z_p)} \THH(\sT_{\sO_\Cu};\Z_p)$ with Proposition~\ref{THHOCbSperf}, we see the following Proposition.
\begin{prop}\label{THHolkperf}
For a smooth proper $\sO_K$-linear category $\sT$, $\THH(\sT_{\ol{k}};\Z_p)$ is perfect in $\Mod_{\THH({\ol{k}};\Z_p)}(\Sp^{BS^1})$.
\end{prop}

Let us recall the calculation of $\TP(\sO_\Cu;\Z_p)$ \cite{BMS2}. On homotopy groups,
\begin{equation}
\varphi^{hS^1}_{\sO_\Cu}:  \pi_* \TCn(\sO_\Cu;\Z_p) \to \pi_*\TP(\sO_\Cu;\Z_p)
\end{equation}
is given by 
\[
\Ainf[u,v]/(uv-\xi) \to \Ainf[\sigma^\pm]
\]
which is $\varphi_{\Ainf}$-linear and sends $u$ to $\sigma$ and $v$ to $\varphi_{\Ainf}(\xi)\sigma^{-1}$, where $u$ and $\sigma$ have degree $2$ and $v$ has degree $-2$. Similarly, on homotopy groups,
\begin{equation}
\varphi^{hS^1}_{\ol{k}}: \pi_* \TCn(\ol{k};\Z_p) \to \pi_*\TP(\ol{k};\Z_p)
\end{equation}
is given by 
\[
W(\ol{k})[u,v]/(uv-p) \to W(\ol{k})[\sigma^\pm]
\]
which is $\varphi_{W(\ol{k})}$-linear and sends $u$ to $\sigma$ and $v$ to $p\sigma^{-1}$.

\begin{thm}\label{TPTPOColkcomp}
Let $\sT$ be a smooth proper $\sO_K$-linear category. The morphism~\eqref{TPTPOColkmap} induces an isomorphism
\[
\pi_i\TP(\sT_{\sO_\Cu};\Z_p)[\frac{1}{p}] \otimes_{\Ainf[\frac{1}{p}]} W(\ol{k})[\frac{1}{p}] \simeq \pi_i\TP(\sT_{\ol{k}};\Z_p)[\frac{1}{p}]
\]
for any $i$. In particular, $\pi_i\TP(\sT_{\sO_\Cu};\Z_p)[\frac{1}{p}]\otimes_{\Ainf[\frac{1}{p}]} W(\ol{k})[\frac{1}{p}]$ is fixed by $G_{K^{ur}}$.
\end{thm}
\begin{proof}
The morphism~\eqref{TPTPOColkmap} yields a morphism
\begin{equation}\label{TPTPOColkmap2}
\TP(\sT_{\sO_\Cu};\Z_p)[\frac{1}{p}]\otimes_{\TP(\sO_\Cu;\Z_p)[\frac{1}{p}]} \TP(\ol{k};\Z_p)[\frac{1}{p}] \to \TP(\sT_{\ol{k}};\Z_p)[\frac{1}{p}].
\end{equation}
 By Proposition~\ref{THHOCbSperf} and Proposition~\ref{THHolkperf}, both sides of the map \eqref{TPTPOColkmap} yield symmetric monoidal functors from $\Catsat(\sO_K)$ to $\Mod_{\TP(\ol{k};\Z_p)}(\Sp)$, and the map \eqref{TPTPOColkmap2} yields a symmetric monoidal natural transformation between them. Thus the morphism \eqref{TPTPOColkmap2} is an equivalence. By Theorem~\ref{BKtheorem} and \cite[Proposition 4.3]{BMS1}, there is $n$ such that for any $i\geq n$, the homotopy group $\pi_i\TCn(\sT/\bS[z];\Z_p)$ is a Breuil-Kisin module, and $\pi_i\TCn(\sT/\bS[z];\Z_p)[\frac{1}{p}]$ is a finite free $\mS[\frac{1}{p}]$-module. By Theorem~\ref{TPTCcomp} and the fact that $\TP(\sT_{\sO_\Cu};\Z_p)$ is $2$-periodic, $\pi_i\TP(\sT_{\sO_\Cu};\Z_p)[\frac{1}{p}]$ is a finite free $\Ainf[\frac{1}{p}]$-module for any $i$. Thus $\pi_*\TP(\sT_{\sO_\Cu};\Z_p)[\frac{1}{p}]$ is a flat graded $\pi_*\TP(\sO_\Cu;\Z_p)[\frac{1}{p}]$-module and we obtain the claim.
\end{proof}

\begin{lemma}\label{periodlemma2}
Let $\sT$ be a smooth proper $\sO_K$-linear category. Then the followings hold:
\begin{itemize}
  \item[(1)] $\pi_*\TCn(\sT_{\sO_\Cu};\Z_p)$ is a finitely generated $\pi_*\TCn(\sO_\Cu/\bS[z];\Z_p)$-module, thus $\pi_i\TCn(\sT_{\sO_\Cu};\Z_p)$ is a finitely generated $\Ainf$-module for any $i$.
  \item[(2)] There is a natural number $n$ satisfying that for any $j \geq n$, $\pi_j\TCn(\sT_{\sO_\Cu};\Z_p)\overset{u\cdot }{\to} \pi_{j+2}\TCn(\sT_{\sO_\Cu};\Z_p)$ is an isomorphism. 
\end{itemize}
\end{lemma}
\begin{proof}
Claim (1) directly follows from Proposition ~\ref{THHOCbSperf}. For any $j \geq 0$, by the calculation of $\pi_* \TCn({\sO_\Cu};\Z_p)$ (see \cite[Proposition 11.10]{BMS2}), we know $\pi_j\TCn({\sO_\Cu};\Z_p)\overset{u\cdot}{\to} \pi_{j+2}\TCn({\sO_\Cu};\Z_p)$ is an isomorphism. By claim (1), we see claim (2).
\end{proof}

\begin{thm}\label{TPTCnOCOCcomp}
Let $\sT$ be a smooth proper $\sO_K$-linear category. There is a natural number $n$ such that the cyclotomic Frobenius morphism $\varphi^{hS^1}_{\sT_{\sO_\Cu}}:\TCn(\sT_{\sO_\Cu};\Z_p) \to \TP(\sT_{\sO_\Cu};\Z_p)$ induces a $G_K$-equivariant isomorphism
\begin{equation}\label{TCnTPOCAinfequi}
\pi_i\TCn(\sT_{\sO_\Cu};\Z_p)\otimes_{\Ainf,\varphi_{\Ainf}} \Ainf \simeq \pi_i\TP(\sT_{\sO_\Cu};\Z_p)
\end{equation}
for any $i\geq n$. 
\end{thm}
\begin{proof} Cyclotomic Frobenius map $\varphi^{hS^1}_{\sT_{\sO_\Cu}}:\TCn(\sT_{\sO_\Cu};\Z_p) \to \TP(\sT_{\sO_\Cu};\Z_p)$ yields a morphism 
\[
\underline{\varphi}^{hS^1}_{\sT_{\sO_\Cu}}:\TCn(\sT_{\sO_\Cu};\Z_p)[\frac{1}{u}] \to \TP(\sT_{\sO_\Cu};\Z_p).
\]
The morphism $\underline{\varphi}^{hS^1}_{\sT_{\sO_\Cu}}$ induces a morphism of $\TP(\sO_\Cu;\Z_p)$-module: 
  \begin{equation}\label{TCnTPOCequi}
  \TCn(\sT_{\sO_\Cu};\Z_p)[\frac{1}{u}]\otimes_{\TCn(\sO_\Cu;\Z_p)[\frac{1}{u}],\underline{\varphi}^{hS^1}_{\sO_\Cu}} \TP(\sO_\Cu;\Z_p)\to \TP(\sT_{\sO_\Cu};\Z_p).
  \end{equation}
By Proposition~\ref{THHOCbSperf}, both sides of the map \eqref{TCnTPOCequi} yield symmetric monoidal functors from $\Catsat(\sO_K)$ to $\Mod_{\TP(\sO_\Cu;\Z_p)}(\Sp)$, thus the map \eqref{TCnTPOCequi} is an equivalence. On homotopy groups, $\underline{\varphi}^{hS^1}_{\sO_\Cu}: \pi_*\TCn(\sO_\Cu;\Z_p)[\frac{1}{u}] \to \pi_* \TP(\sO_\Cu;\Z_p)$ is given by $\varphi_{\Ainf}$-linear map
  \[
  \Ainf[u^\pm] \to \Ainf[\sigma^\pm],
  \]
  thus this is an isomorphism. Thus we obtain the isomorphism
  \[
  \pi_i \TCn(\sT_{\sO_\Cu};\Z_p)[\frac{1}{u}] \otimes_{\Ainf,\varphi_{\Ainf}} \pi_i \TP(\sT_{\sO_\Cu};\Z_p)
  \]
  for any $i$. By Lemma~\ref{periodlemma2}, there is a natural number $n$ such that $\pi_i \TCn(\sT_{\sO_\Cu};\Z_p)\simeq \pi_i \TCn(\sT_{\sO_\Cu};\Z_p)[\frac{1}{u}]$ for any $i \geq n$.
\end{proof}

For a smooth proper $\sO_K$-linear category $\sT$, we will prove that the free Breuil-Kisin module $\pi_i \TCn(\sT/\bS[z];\Z_p)^{\vee}$ is a Breuil-Kisin $G_K$-module.

\begin{thm}\label{BKGKmodulethm}
  Let $\sT$ be a smooth proper $\sO_K$-linear category. There is a $n$ such that for any $i\geq n$, the followings holds:
  \begin{itemize}
    \item[(1)] There is a continuous $\Ainf$-semi-linear $G_K$-action on 
    \[
\widehat{\pi_i\TCn(\sT/\bS[z];\Z_p)^{\vee}}=\pi_i\TCn(\sT/\bS[z];\Z_p)^{\vee}\otimes_{\mS,\phi}\Ainf.
    \]
    \item[(2)] $G_K$ commutes with $\widehat{\varphi}_{\pi_i \TCn(\sT/\bS[z];\Z_p)^{\vee}}$.
    \item[(3)] $\pi_i \TCn(\sT/\bS[z];\Z_p)^{\vee} \subset \bigl(\widehat{\pi_i \TCn(\sT/\bS[z];\Z_p)^{\vee}}\bigr)^{G_{K_\infty}}$ via the embedding $\pi_i \TCn(\sT/\bS[z];\Z_p)^{\vee} \subset \widehat{\pi_i \TCn(\sT/\bS[z];\Z_p)^{\vee}}$. 
    \item[(4)] $\pi_i\TCn(\sT_{\sO_\Cu};\Z_p)^{\vee}[\frac{1}{p}]\otimes_{\Ainf[\frac{1}{p}]} W(\ol{k})[\frac{1}{p}]$ is fixed by $G_{K^{ur}}$.
  \end{itemize}
  In particular, $\pi_i \TCn(\sT/\bS[z];\Z_p)^{\vee}$ is a Breuil-Kisin $G_K$-module.
\end{thm}
\begin{proof}
After localization by $\xi\in \Ainf \simeq \pi_0 \TCn({\sO_{\Cu}};\Z_p)$, the morphism $\can_\sT$ induces a morphism of $\TP(\sO_\Cu;\Z_p)[\frac{1}{\xi}]$-module:
\begin{equation}\label{TCTPOCcan}
\TCn(\sT_{\sO_{\Cu}};\Z_p)[\frac{1}{\xi}] \otimes_{\TCn({\sO_{\Cu}};\Z_p)[\frac{1}{\xi}],\can_{{\sO_{\Cu}}}} \TP({\sO_{\Cu}};\Z_p)[\frac{1}{\xi}] \to \TP(\sT_{\sO_{\Cu}};\Z_p)[\frac{1}{\xi}].
\end{equation}
Both sides of the map \eqref{TCTPOCcan} yield symmetric monoidal functors from $\Catsat(\sO_K)$ to $\Mod_{\TP(\sO_{\sO_{\Cu}};\Z_p)[\frac{1}{\xi}]}(\Sp)$, and the map \eqref{TCTPOCcan} yields a symmetric monoidal natural transformation between them. Thus the morphism \eqref{TCTPOCcan} is an equivalence. Note that the morphism 
\[
\can_{{\sO_{\Cu}}}:\pi_*\TCn({\sO_{\Cu}};\Z_p)[\frac{1}{\xi}] \to \pi_*\TP({\sO_{\Cu}};\Z_p)[\frac{1}{\xi}]
\]
is an isomorphism (see \cite[Proposition 11.10]{BMS2}). This yields a $G_K$-equivariant isomorphism
\begin{equation}\label{canisom2}
\can_{\sT_{\sO_{\Cu}}}[\frac{1}{\xi}]:\pi_*\TCn(\sT_{\sO_{\Cu}};\Z_p)[\frac{1}{\xi}] \simeq \pi_*\TP(\sT_{\sO_{\Cu}};\Z_p)[\frac{1}{\xi}].
\end{equation}
Combined with \eqref{TCnTPOCAinfequi}, this yields a $G_K$-equivariant isomorphism
\begin{equation}\label{FrobTCnOC}
\pi_i\TCn(\sT_{\sO_\Cu};\Z_p)\otimes_{\Ainf,\varphi_{\Ainf}} \Ainf[\frac{1}{\xi}]\overset{\eqref{TCnTPOCAinfequi}}{\simeq} \pi_i\TP(\sT_{\sO_\Cu};\Z_p)[\frac{1}{\xi}] \overset{\eqref{canisom2}}{\simeq}\pi_i\TCn(\sT_{\sO_{\Cu}};\Z_p)[\frac{1}{\xi}]
\end{equation}
for any $i \geq n$. Since $\varphi_{\Ainf}$ is an isomorphism, by Theorem~\ref{TPTCcomp} and Theorem~\ref{TPTCnOCOCcomp}, we have a $G_{K_\infty}$-equivariant isomorphism
\begin{equation}\label{BIWAKOUP}
\pi_i\TCn(\sT/\bS[z];\Z_p)\otimes_{\mS,\phi}\Ainf\simeq \pi_i\TCn(\sT_{\sO_\Cu};\Z_p),
\end{equation}
 and we see the isomorphism \eqref{FrobTCnOC} is same as $\widehat{\varphi}_{\pi_i\TCn(\sT/\bS[z];\Z_p)}$. $G_K$ naturally acts on $\pi_i\TCn(\sT_{\sO_\Cu};\Z_p)$, which is $\pi_0 \TCn(\sO_{\Cu};\Z_p)\simeq \Ainf$-semi-linear. Since $\varphi:\mS\to\Ainf$ is flat, the dual of \eqref{BIWAKOUP} becomes the following $G_{K_\infty}$-equivariant isomoprhism
 \begin{equation}\label{kenka1} \pi_i\TCn(\sT/\bS[z];\Z_p)^{\vee}\otimes_{\mS,\phi}\Ainf\simeq \pi_i\TCn(\sT_{\sO_\Cu};\Z_p)^{\vee}.
 \end{equation}
Via this isomorphism, $\pi_i\TCn(\sT/\bS[z];\Z_p)^{\vee}\otimes_{\mS,\phi}\Ainf$ admits a continuous $\Ainf$-semi-linear $G_K$-action.
 
 Since the isomorphism~\eqref{FrobTCnOC} is $G_K$-equivariant, $G_K$ commutes with $\widehat{\varphi}_{\pi_i \TCn(\sT/\bS[z];\Z_p)^{\vee}}$ and we obtain the claim (1) and (2). 
 
Consider the $G_{K_\infty}$-equivariant isomorphism \eqref{kenka1}. Since $g\in G_{K_\infty}$ acts on $1\otimes g$ on $\pi_i\TCn(\sT/\bS[z];\Z_p)^{\vee}\otimes_{\mS,\phi}\Ainf$, we obtain the claim (3).
 
Since $\varphi_{\Ainf}:\Ainf \to \Ainf$ is flat, we have $G_K$-equivariant isomorphism
\begin{equation}\label{freeTCTP}
\pi_i\TCn(\sT_{\sO_\Cu};\Z_p)^{\vee}\otimes_{\Ainf,\varphi_{\Ainf}} \Ainf \simeq \pi_i\TP(\sT_{\sO_\Cu};\Z_p)^{\vee}
\end{equation}
by Theorem~\ref{TPTCnOCOCcomp}. Combined with Theorem~\ref{TPTPOColkcomp}, we see that $\pi_i\TCn(\sT_{\sO_\Cu};\Z_p)^{\vee}[\frac{1}{p}] \otimes_{\Ainf[\frac{1}{p}]} W(\ol{k})[\frac{1}{p}]$ is fixed by $G_{K^{ur}}$.

Since $\pi_i\TCn(\sT_{\sO_\Cu};\Z_p)^{\vee}$ is finite free $\mS$-module, there is a natural inclusion $\pi_i\TCn(\sT_{\sO_\Cu};\Z_p)^{\vee} \hookrightarrow\pi_i\TCn(\sT_{\sO_\Cu};\Z_p)^{\vee}[\frac{1}{p}]$ and it induces a natural inclusion
\[
\pi_i\TCn(\sT_{\sO_\Cu};\Z_p)^{\vee} \otimes_{\Ainf} W(\ol{k}) \hookrightarrow \pi_i\TCn(\sT_{\sO_\Cu};\Z_p)^{\vee}[\frac{1}{p}] \otimes_{\Ainf[\frac{1}{p}]} W(\ol{k})[\frac{1}{p}].
\]
Thus $\pi_i\TCn(\sT_{\sO_\Cu};\Z_p)^{\vee} \otimes_{\Ainf} W(\ol{k})$ is fixed by $G_{K^{ur}}$. The last claim follows from claims (1)-(4) and \cite[Lemma 7.1.2]{Gao}.
\end{proof}

For a Breuil-Kisin module $(\fM,\varphi_\fM)$, let $\ol{\varphi_{\fM}}$ denote the composite 
\[
\ol{\varphi_{\fM}}: \fM\otimes_{\mS} W(\Cu^{\flat}) \overset{\varphi_{\fM}'\otimes \varphi_{\Cu}}{\to} \fM[\frac{1}{E}]\otimes_{\mS[\frac{1}{E}]}W(\Cu^{\flat}) =\fM\otimes_{\mS} W(\Cu^{\flat})
\]
where $\varphi_{\fM}'$ is the composite $\fM \overset{\varphi}{\to} \fM\otimes_{\mS,\varphi}\mS[\frac{1}{E}] \overset{\varphi_{\fM}}{\to} \fM[\frac{1}{E}]$, and $\varphi_{\Cu}$ is the Frobenius of $W(\Cu^{\flat})$. Let us define a $\Z_p$-module:
\[
T_{\Ainf}(\fM):=\bigl(\widehat{\fM}\otimes_{\Ainf} W(\Cu^{\flat})\bigr)^{\ol{\varphi_{\fM}}=1}.
\]
If $(\fM,\varphi_\fM)$ is equipped with a Breuil-Kisin $G_K$-module structure, then $T_{\Ainf}(\fM)$ inherits a natural structure of a $\Z_p[G_K]$-module.

\begin{thm}\label{mainwithproof}
\begin{itemize}
  \item[(1)] For a smooth proper $\sO_K$-linear category $\sT$ and a sufficiently large $i$, the $\Z_p[G_K]$-module $T_{\Ainf}(\pi_i\TCn(\sT/\bS[z];\Z_p)^{\vee})$ is a $\Z_p$-lattice of a semi-stable representation.

  \item[(2)] We assume $\sT_{\Cu}$ admits a geometric realization. Then there is a $G_K$-equivariant isomorphism
\[
\pi_i \LK K(\sT_{\Cu})^{\vee} \simeq T_{\Ainf}(\pi_i\TCn(\sT/\bS[z];\Z_p)^{\vee}).
\]
In particular, $\pi_i \LK K(\sT_{\Cu})\otimes_{\Z_p}\Q_p$ is a semi-stable representation.
\end{itemize}
\end{thm}
\begin{proof} (1)  By Theorem~\cite[Theorem 7.1.7]{Gao}, we obtain the claim. 

(2) Let $\sT$ be a smooth proper $\sO_K$-linear category. We assume $\sT_{\Cu}$ admits a geometric realization. By the same discussion as the proof of Theorem~\ref{KTCncomp}, we have an equivalence
\begin{equation}\label{KTCnequi2}
\LK K(\sT_\Cu)\otimes_{\LK K(\Cu)} \LK (\TCn(\sO_\Cu;\Z_p)[\frac{1}{\xi}]) \to \LK (\TCn(\sT_{\sO_\Cu};\Z_p)[\frac{1}{\xi}]). 
\end{equation}
By the functionality of cyclotomic Frobenius $\varphi_{-}^{hS^1}:\TCn(-;\Z_p) \to \TP(-;\Z_p)$, we have the following commutative diagram
\[
\xymatrix{
\LK K(\sT_\Cu)\otimes_{\LK K(\Cu)} \LK (\TCn(\sO_\Cu;\Z_p)[\frac{1}{\xi}]) \ar[r]_-{\simeq}^-{\eqref{KTCnequi2}} \ar[d]_-{\text{id}\otimes \varphi_{{\sO_{\Cu}}}^{hS^1}} & \LK (\TCn(\sT_{\sO_\Cu};\Z_p)[\frac{1}{\xi}]) \ar[d]^-{ \varphi_{\sT_{\sO_{\Cu}}}^{hS^1}}\\
\LK K(\sT_\Cu)\otimes_{\LK K(\Cu)} \LK (\TP(\sO_\Cu;\Z_p)[\frac{1}{\tilde{\xi}}])  \ar[r]^-{\simeq}_-{\eqref{KTCnequi}}& 
 \LK (\TP(\sT_{\sO_\Cu};\Z_p)[\frac{1}{\tilde{\xi}}])
 }
\]
We note $\xi | \mu$ in $\Ainf$. Since $\varphi_{\sO_{\Cu}}^{hS^1}:\pi_0\LK\TCn(\sO_{\Cu};\Z_p)[\frac{1}{\xi}]\to \pi_0\LK \TP(\sO_{\Cu};\Z_p)[\frac{1}{\tilde{\xi}}]$ is given by the Frobenius $\varphi_{\Ainf}:\Ainf[\frac{1}{\mu}]^{\wedge}_p \to \Ainf[\frac{1}{\varphi(\mu)}]^{\wedge}_p$, on homotopy groups, we have $G_K$-equivariant diagram
\[
\xymatrix{
\pi_i\LK K(\sT_\Cu)\otimes_{\Z_p} \Ainf[\frac{1}{\mu}]^{\wedge}_p \ar[r]^-{\simeq}\ar[d]_-{\text{id}\otimes \varphi_{\Ainf}} & \pi_i\TCn(\sT_{\sO_\Cu};\Z_p)[\frac{1}{\mu}]^{\wedge}_p \ar[d]^-{ \varphi_{\sT_{\sO_{\Cu}}}^{hS^1}} \\
\pi_i\LK K(\sT_\Cu)\otimes_{\Z_p} \Ainf[\frac{1}{\varphi(\mu)}]^{\wedge}_p  \ar[r]_-{\simeq}& 
 \pi_i\TP(\sT_{\sO_\Cu};\Z_p)[\frac{1}{\varphi(\mu)}]^{\wedge}_p 
 }
\]
After the base change along $\Ainf[\frac{1}{\mu}]^{\wedge}_p \to W(\Cu^{\flat})$, the above diagram becomes the following $G_K$-equivariant diagram.
\begin{equation}\label{GKequidiagram}
\xymatrix{
\pi_i\LK K(\sT_\Cu)\otimes_{\Z_p} W(\Cu^{\flat}) \ar[r]^-{\simeq}\ar[d]_-{\text{id}\otimes \varphi_{\Cu}} & \pi_i\TCn(\sT_{\sO_\Cu};\Z_p)\otimes_{\Ainf}W(\Cu^{\flat}) \ar[d]^-{ \varphi_{\sT_{\sO_{\Cu}}}^{hS^1}\otimes \varphi_{\Cu}} \ar[rd]^{\widehat{\varphi}\otimes \varphi_{\Cu} }& \\
\pi_i\LK K(\sT_\Cu)\otimes_{\Z_p} W(\Cu^{\flat})  \ar[r]_-{\simeq}& 
     \pi_i\TP(\sT_{\sO_\Cu};\Z_p)\otimes_{\Ainf}W(\Cu^{\flat}) & \pi_i\TCn(\sT_{\sO_\Cu};\Z_p)\otimes_{\Ainf}W(\Cu^{\flat}) \ar[l]_-{\can_{\sT_{\sO_{\Cu}}}}^-{\simeq}
 }
\end{equation}
Thus the isomoprhism $\pi_i\LK K(\sT_\Cu)\otimes_{\Z_p} W(\Cu^{\flat})\simeq \pi_i\TCn(\sT_{\sO_\Cu};\Z_p)\otimes_{\Ainf}W(\Cu^{\flat})$ is Frobenius-equivariant. Since $\Z_p\to \Ainf$ and $\Ainf \to W(\Cu^\flat)$ are flat and compatible with Frobenius automorphisms, the dual of the isomorphism becomes the Frobenius-equivariant isomorphism
\[
\pi_i\LK K(\sT_\Cu)^{\vee}\otimes_{\Z_p} W(\Cu^{\flat})\simeq \pi_i\TCn(\sT_{\sO_\Cu};\Z_p)^{\vee}\otimes_{\Ainf}W(\Cu^{\flat})
\]
and we obtain the claim.
\end{proof}

\begin{remark}\label{rem38}
Similarly, we can obtain that there is an isomorphism of $\Z_p[G_K]$-modules:
\[
\pi_i \LK K(\sT_{\Cu})^{\vee\vee} \simeq T_{\Ainf}(\pi_i\TCn(\sT/\bS[z];\Z_p)^{\vee\vee}).
\]
\end{remark}

\section{{$(\varphi,\hat{G})$}-modules and crystalline representations}\label{sec4}

In this section, we will prove Theorem~\ref{maininintro}. We would like to thank H.Gao \cite{lettertokeiho} for sharing the strategy and key constructions to prove Theorem~\ref{maininintro}. 
 \subsection{Breuil-Kisin {$G_K$}-modules and {$(\varphi,\hat{G})$}-modules}
 Let us recall $(\varphi,\hat{G})$-module from \cite{DuLiu}. Let $L:=\bigcup_{n=1}^{\infty}K_{\infty}(\zeta_{{n}})$, $\hat{G}:=\text{Gal} (L/K)$ and $H_K:=\text{Gal}(L/K_{\infty})$. Du-Liu construct $\mS$-algebras $\mStw$ and $\mSsttw$ as the following \cite[section 2.1 and 2.3]{DuLiu}. We set $\mS^{\hat{\otimes} 2}:= \mS[[y-z]]=W[[z,y-z]]$, $\mS^{\hat{\otimes} 2}$ is $\mS\otimes_{\Z_p} \mS$-algebra via $z\otimes 1 \mapsto z$, $1\otimes z \mapsto y=(y-z)+z$, in this way, one can extend Frobenius action on $\mS$ to $\mS^{\hat{\otimes} 2}$ which is Frobenius on $W$ and  sends $z$ to $z^p$ and $y-z$ to $y^p-z^p$. Let $\mS^{(2)}$ be $\mS[[y-z]]\{\frac{y-z}{E}\}^{\wedge}_{\delta}$, where $\{-\}^{\wedge}_{\delta}$ means
    freely adjoining elements in the category of $(p,E)$-completed $\delta$-$\mS$-algebras \cite[section 2.1 and 4.1]{DuLiu}. $\mS^{(2)}$ is a $\mS$-algebra via $u \mapsto u$, and the structure map is flat (see \cite[Proposition 2.2.7]{DuLiu}), and is a sub-$\mS$-algebra of $\Ainf$, where we regard $\Ainf$ is as a $\mS$-algebra via $\phi$ (see \cite[section 2.4]{DuLiu}). Let us recall the definition of $\mSsttw$ from \cite[section 2.3]{DuLiu}. Define a Frobenius action on $W[[z,\mathfrak{y}]]$ which sends $x$ to $x^p$ and $\mathfrak{y}$ to $(1+\mathfrak{y})^p-1$ and set $\mathfrak{w}=\frac{\mathfrak{y}}{E}$. Let $\mSsttw=\mS[[z,\mathfrak{y}]]\{\mathfrak{w}\}^{\wedge}_{\delta}$, and it is a sub-$\mS$-algebra of $\Ainf$, where we regard $\Ainf$ is as a $\mS$-algebra via $\phi$ (see \cite[section 2.4]{DuLiu}). There is a natural inclusion of $\mS$-algebra $W[[z,y-z]] \subset W[[z,\mathfrak{y}]]$ which sends $y$ to $z(\mathfrak{y}+1)$, it is $\delta$-ring map, by the constructions we have an inclusion of sub-$\mS$-algebras $\mStw \hookrightarrow \mSsttw$ of $\Ainf$. Let us denote $\theta_0:\mS \to \mStw~|~z\mapsto z$, and denote $\theta_1:\mS \to \mStw~|~z\mapsto y$. We note that there is a commutative diagram (see \cite[Corollary 2.4.5]{DuLiu}):
    \begin{equation}\label{mSmS2Ainf}
    \xymatrix{
    \mS \ar@{^{(}-_>}[d]_-{\theta_0} \ar@{^{(}-_>}[rrd]^{\phi} & & \\
    \mStw \ar@{^{(}-_>}[r] \ar@/_18pt/[rr]_{\iota} & \mSsttw \ar@{^{(}-_>}[r] & \Ainf 
    }
    \end{equation}
where $\iota$ sends $z$ to $[\pi^{\flat}]$ and $y$ to $[\epsilon]\cdot [\pi^{\flat}]$. We regard $\mStw$ and $\mSsttw$ as $\mS$-algebras via the diagram unless otherwise stated. 
    
\begin{defn}[{\cite[Definition 3.3.2]{DuLiu}}]
    Let $(\fM,\varphi_{\fM})$ be a finite free Breuil-Kisin module of finite $E$-height. We call $(\fM,\varphi_{\fM})$ $(\varphi,\hat{G})$-module if it satisfies the following conditions.
    \begin{itemize}
        \item[(1)] There is a continuous $\mSsttw$-semi-linear $\hat{G}$-action on $\fM\otimes_{\mS}\mSsttw$.
        \item[(2)] $\hat{G}$ commutes with $\varphi$ on $\fM\otimes_{\mS}\mSsttw$.
        \item[(3)] $\fM \subset (\fM\otimes_{\mS}\mSsttw)^{H_K}$.
        \item[(4)] $\hat{G}$ acts on $(\fM\otimes_{\mS}\mSsttw) \otimes_{\mSsttw}W(k)$ trivially. 
    \end{itemize}
\end{defn}

\begin{remark}\label{remark4.2} For a Breuil-Kisin $G_K$-module $(\fM,\varphi_{\fM},G_K\curvearrowright \fM\otimes_{\mS,\phi}\Ainf)$, the $p$-adic representation $T_{\Ainf}(\fM)$ is a $\Z_p$-lattice of semi-stable representation of non-negative Hodge-Tate weight \cite[Theorem 7.1.7]{Gao}. Du-Liu proved that the sub-modules $\fM\otimes_{\mS}\mSsttw \subset \fM \otimes_{\mS,\phi} \Ainf \subset  T_{\Ainf}(\fM)\otimes_{\Z_p}\Ainf$ are $G_K$-stable \cite[Theorem 3.3.3]{DuLiu} (see also \cite[Proposition 3.1.3]{Liu10}). By Definition \ref{BKGKdefn} (3), $G_L=\text{Gal}(\ol{K}/L) (\subset G_{K_\infty})$ acts on $\fM\subset \fM\otimes_{\mS,\phi}\Ainf$ trivially, thus $G_{K}$-action on $\fM\otimes_{\mS}\mSsttw$ factors through $\hat{G}$. The $\hat{G}$-action on $\fM\otimes_{\mS}\mSsttw$ induces the $(\varphi,\hat{G})$-module structure on $(\fM,\varphi_{\fM})$ \cite[Theorem 3.3.3]{DuLiu}. 
\end{remark}

\subsection{The comparison theorem between $\TP(\sT/\bS[z];\Z_p)$ and $\TP(\sT/\bS[z_0,z_1];\Z_p)$} We regard $\sO_K$ as an $\bS[z_0,z_1]$-algebra via the map $z_0,z_1 \mapsto \pi$. In \cite{LiuWang}, Liu-Wang revealed the structure of $\TP(\sO_K/\bS[z_0,z_1];\Z_p)$: there is an isomorphism 
\[
\pi_0\TP(\sO_K/\bS[z_0,z_1];\Z_p)\simeq \mStw,
\]
where we identify $z_0=z$ and $z_1=y$ (see \cite[Theorem 1.3]{LiuWang}). In this section, we will compare $\TP(\sT/\bS[z];\Z_p)$ with $\TP(\sT/\bS[z_0,z_1];\Z_p)$ for an $\sO_K$-linear category $\sT$. We note that the functor 
\[
\TP(-/\bS[z_0,z_1];\Z_p):\Cat(\sO_K) \overset{\THH(-/\bS[z_0,z_1];\Z_p)}{\to} \Sp^{BS^1} \overset{(-)^{tS^1}}{\to} \Sp
\]
is a lax symmetric monoidal, $\THH(\sO_K/\bS[z_0,z_1];\Z_p)$ admits an $\mathbb{E}_\infty$-ring structure and $\TP(\sO_K/\bS[z_0,z_1];\Z_p)$ also admits an $\mathbb{E}_\infty$-ring structure. For $i=0,1$, the map of ring spectra $\bS[z] \overset{z\mapsto z_i}{\to} \bS[z_0,z_1]$ induces a morphism of $\mathbb{E}_\infty$-ring spectra
\begin{equation}\label{TPringsp}
e_i:\TP(\sO_K/\bS[z];\Z_p) \overset{z\mapsto z_i}{\to} \TP(\sO_K/\bS[z_0,z_1];\Z_p).
\end{equation}
For $\heartsuit \in \{\THH , \TP , \TCn \}$ and an $\sO_K$-linear category $\sT$, the left unit $e_{0,\sT}$ and right unit $e_{1,\sT}$ are the maps
\[
\heartsuit (\sT/\bS[z];\Z_p) \to \heartsuit (\sT/\bS[z_0,z_1];\Z_p)
\]
induced by $z\mapsto z_0$ and $z\mapsto z_1$ respectively. Let us denote by $\THH(\bS[z_0,z_1]/{}_0\bS[z])$ (resp. $\THH(\bS[z_0,z_1]/{}_1\bS[z])$) the relative topological Hochschild homology of $\bS[z] \to \bS[z_0,z_1] ~|~z \to z_0$ (resp. $\bS[z] \to \bS[z_0,z_1] ~|~z \to z_1$). There is a commutative diagram in $\Mod_{\THH(\bS[z])}(\Sp^{BS^1})$:
\[
\xymatrix{
&\bS[z_0,z_1]^{\wedge}_p\ar[r] & \THH(\sT/\bS[z_0,z_1];\Z_p)\\
\bS[z]^{\wedge}_p \ar[r]^-{z\mapsto z_i}& \THH(\bS[z_0,z_1]/{}_i\bS[z];\Z_p)\ar[r] \ar[u] & \THH(\sT/\bS[z];\Z_p)\ar[u]_{e_{i,\sT}}\\
\THH(\bS[z];\Z_p)\ar[r]_-{e_i} \ar[u]& \THH(\bS[z_0,z_1];\Z_p)\ar[r] \ar[u]& \THH(\sT;\Z_p) \ar[u]
}
\]
Since the bottom-left square and the bottom square are pushout squares (see \cite[Lemma~2.3]{LiuWang}), 
Similarly since the right square is pushout square, the top-right square is also pushout square. Thus we have the following equivalence (transitivity property of relative $\THH$)
\begin{equation}\label{THHtransitivity}
\THH(\sT/\bS[z_0,z_1];\Z_p) \simeq \THH(\sT/\bS[z];\Z_p)\otimes_{\THH(\bS[z_0,z_1]/{}_i\bS[z];\Z_p)} \bS[z_0,z_1]^{\wedge}_p.
\end{equation}
We also have the following commutative diagram:
\[
\xymatrix{
\THH(\sT/\bS[z];\Z_p)\ar[r]^{e_{i,\sT}}& \THH(\sT/\bS[z_0,z_1];\Z_p)\\
\THH(\sO_K/\bS[z];\Z_p)\ar[r]^-{e_i} \ar[u]& \THH(\sO_K/\bS[z_0,z_1];\Z_p)\ar[u] \\
\THH(\bS[z_0,z_1]/{}_i\bS[z];\Z_p)\ar[r] \ar[u]& \bS[z_0,z_1]^{\wedge}_p\ar[u]
}
\]
Applying $\sT=\perf(\sO_K)$ for the equivalence~\eqref{THHtransitivity}, then we see the bottom and total square are pushout squares. We know that the top square is a pushout square. Thus we have the following $S^1$-equivariant equivalence
\begin{equation}\label{THHtransivity2}
\THH(\sT/\bS[z_0,z_1];\Z_p) \simeq \THH(\sT/\bS[z];\Z_p) \otimes_{\THH(\sO_K/\bS[z];\Z_p)}\THH(\sO_K/\bS[z_0,z_1];\Z_p).
\end{equation}
Notice an exact symmetric monoidal $\infty$-functor
\[
-\otimes_{\THH(\sO_K/\bS[z];\Z_p),e_i}\THH(\sO_K/\bS[z_0,z_1];\Z_p)
\]
from $\Mod_{\THH(\sO_K/\bS[z];\Z_p)}(\Sp^{BS^1})$ to $\Mod_{\THH(\sO_K/\bS[z_0,z_1];\Z_p)}(\Sp^{BS^1})$.
If $\sO_K$-linear category $\sT$ is smooth proper, $\THH(\sT/\bS[z];\Z_p)$ is dualizable and perfect in $\Mod_{\THH(\sO_K/\bS[z];\Z_p)}(\Sp^{BS^1})$. Since the functor $-\otimes_{\THH(\sO_K/\bS[z];\Z_p)}\THH(\sO_K/\bS[z_0,z_1];\Z_p)$ is an exact symmetric monoidal, we see the following lemma:
\begin{lemma}\label{perfTHHz1z2}
For a smooth proper $\sO_K$-linear category $\sT$, $\THH(\sT/\bS[z_0,z_1];\Z_p)$ is dualizable and perfect in $\Mod_{\THH(\sO_K/\bS[z_0,z_1];\Z_p)}(\Sp^{BS^1})$.
\end{lemma}

\begin{prop}\label{Kunnethz1z2}
    Let $\sT_1,\sT_2$ be smooth proper $\sO_K$-linear categories. There are equivalences:
    \[
    \TCn(\sT_1/\bS[z_0,z_1];\Z_p) \otimes_{\TCn(\sO_K/\bS[z_0,z_1];\Z_p)}     \TCn(\sT_2/\bS[z_0,z_1];\Z_p)  \simeq \TCn(\sT_1\otimes_{\sO_K} \sT_2/\bS[z_0,z_1];\Z_p),
    \]
    \[
    \TP(\sT_1/\bS[z_0,z_1];\Z_p) \otimes_{\TP(\sO_K/\bS[z_0,z_1];\Z_p)}     \TP(\sT_2/\bS[z_0,z_1];\Z_p)  \simeq \TP(\sT_1\otimes_{\sO_K} \sT_2/\bS[z_0,z_1];\Z_p).
    \]
\end{prop}
\begin{proof}
    By Lemma~\ref{perfTHHz1z2}, one can prove using the same argument as the proof of Proposition~\ref{TPTCkunneth}.
\end{proof}

\begin{prop}\label{TPTCnSzSz1z2}
Let $\sT_1,\sT_2$ be smooth proper $\sO_K$-linear categories. There are equivalences
\[
\TP(\sT/\bS[z];\Z_p) \otimes_{\TP(\sO_K/\bS[z];\Z_p),e_{i}} \TP(\sO_K/\bS[z_0,z_1];\Z_p) 
 \simeq \TP(\sT/\bS[z_0,z_1];\Z_p) 
\]
\[
\TCn(\sT/\bS[z];\Z_p) \otimes_{\TCn(\sO_K/\bS[z];\Z_p),e_{i}} \TCn(\sO_K/\bS[z_0,z_1];\Z_p) 
 \simeq \TCn(\sT/\bS[z_0,z_1];\Z_p) 
\]
for $i=0,1$. 

\end{prop}
\begin{proof}
For $i=0,1$ and $\heartsuit\in \{\TCn,\TP\}$, there is a commutative diagram:
\[
\xymatrix{
\heartsuit(\sT/\bS[z];\Z_p) \ar[r]^-{e_{i,\sT}}\ar[d]& \heartsuit(\sT/\bS[z_0,z_1];\Z_p)\ar[d]\\
\heartsuit(\sO_K/\bS[z];\Z_p) \ar[r]_-{e_i} & \heartsuit(\sO_K/\bS[z_0,z_1];\Z_p)
}
\]
This diagram induces a morphism of $\heartsuit(\sO_K/\bS[z_0,z_1])$-module spectra 
\begin{equation}\label{heartmor}
\heartsuit(\sT/\bS[z];\Z_p) \otimes_{\heartsuit(\sO_K/\bS[z];\Z_p),e_{i}} \heartsuit(\sO_K/\bS[z_0,z_1];\Z_p) 
 \to \heartsuit(\sT/\bS[z_0,z_1];\Z_p).
\end{equation}
By Proposition~\ref{TPTCkunneth} and Proposition~\ref{Kunnethz1z2}, both sides of the map \eqref{heartmor} yield symmetric monoidal functors from $\Catsat(\sO_K)$ to $\Mod_{\heartsuit(\sO_K/\bS[z_0,z_1];\Z_p)}(\Sp)$. Since every smooth proper $\sO_K$-linear category is dualizable in $\Catsat$, by \cite[Proposition 4.6]{KunnethTP} we obtain the claim.\end{proof}

There is an equivalence of $\mathbb{E}_{\infty}$-ring spectra $\text{ex}:\bS[z_0,z_1] \overset{\simeq}{\to} \bS[z_0,z_1] ~|~ z_0 \mapsto z_1, z_1 \mapsto z_0$, and it induces a commutative diagram
\[
\xymatrix{
\TP(\sO_K/\bS[z];\Z_p) \ar[d]_-{e_0} \ar[rd]^-{e_1} & \\
\TP(\sO_K/\bS[z_0,z_1];\Z_p) \ar[r]_-{\text{ex}} & \TP(\sO_K/\bS[z_0,z_1];\Z_p)
}
\]
and an isomorphism of $\mS$-algebra (see also \cite[section 4]{DuLiu})
\begin{equation}\label{koukamS}
\mStw \otimes_{\theta_0,\mS} \mS\simeq \mStw \otimes_{\theta_1,\mS}\mS.
\end{equation}

By \cite[Corollary 3.7]{LiuWang}, we see that $\TP(\sO_K/\bS[z_0,z_1];\Z_p)$ is $2$-periodic and there is an isomorphism $\pi_*\TP(\sO_K/\bS[z_0,z_1];\Z_p)=\mS^{(2)}[\sigma,\sigma^{-1}]$ by choosing a generator $\sigma\in \pi_2 \TP(\sO_K/\bS[z_0,z_1];\Z_p)$. For $i=0,1$, on homotopy groups, the morphism 
\[
\pi_* \TP(\sO_K/\bS[z];\Z_p) \overset{e_i}{\to} \pi_* \TP(\sO_K/\bS[z_0,z_1];\Z_p)
\]
is given by 
\begin{equation}\label{TPstructuremap}
\mS[u,u^{-1}] \to \mS^{(2)}[\sigma,\sigma^{-1}]
\end{equation}
which is $\theta_i$-linear map andcghgfn sends $u$ to $a_i\cdot\sigma$ for some units $a_i\in (\mS^{(2)})^*$. 

\begin{prop}\label{TPmSmS2}
    There is an isomorphism 
\begin{equation}
\pi_j\TP(\sT/\bS[z];\Z_p) \otimes_{\mS,\theta_i} \mS^{(2)} 
 \simeq \pi_j \TP(\sT/\bS[z_0,z_1];\Z_p) 
\end{equation}
for any $j$ and $i=0,1$. 
\end{prop}
\begin{proof}
If $i=0$, by \cite[Proposition 2.2.7]{DuLiu} the graded-ring morphism \eqref{TPstructuremap} is flat. We obtain the claim by Proposition \ref{TPTCnSzSz1z2}. By the isomorphism \eqref{koukamS}, $\theta_1$ is also flat. Thus the morphism \eqref{TPstructuremap} is also flat for $i=1$.
\end{proof}

\subsection{The comparison theorem between $\TP(\sT /\bS[z_0,z_1]; \Z_p)$ and $\TP(\sT_{\sO_{\Cu}}/\bS[z_0^{1/p^{\infty}}, z_1^{1/p^{\infty}}]; \Z_p)$}  We regard $\sO_{\Cu}$ as an $\bS[z_0^{1/p^{\infty}},z_1^{1/p^{\infty}}]$-algebra via the map $z_0^{1/p^{n}} \mapsto \pi^{1/p^{n}}, z_1^{1/p^{n}} \mapsto \zeta_{n}\pi^{1/p^{n}}$. We will prove the comparison theorem between $\TP(\sT /\bS[z_0,z_1]; \Z_p)$ and $\TP(\sT_{\sO_{\Cu}}/\bS[z_0^{1/p^{\infty}}, z_1^{1/p^{\infty}}]; \Z_p)$ for a smooth proper $\sO_K$-linear category $\sT$. The following Lemma is inspired by H. Gao in \cite{lettertokeiho}.

\begin{lemma}\label{BMSTHHlemma2} 
The natural map $\THH(\sT_{\sO_{\Cu}};\Z_p) \to \THH(\sT_{\sO_{\Cu}}/\bS[z_0^{1/p^{\infty}},z_1^{1/p^{\infty}}];\Z_p)$ is an equivalence which is compatible with $S^1$-action. In particular, the natural map
\begin{equation*}\label{TPbSzz}
    \TP(\sT_{\sO_{\Cu}};\Z_p) \to \TP(\sT_{\sO_{\Cu}}/\bS[z_0^{1/p^{\infty}},z_1^{1/p^{\infty}}];\Z_p)
\end{equation*}
is an equivalence. 
\end{lemma}
\begin{proof}
    There is a diagram 
\begin{align*}
\THH(\sT_{\sO_\Cu};\Z_p) 
 &\to \THH(\sT_{\sO_\Cu}/\bS[z_0^{1/p^{\infty}},z_1^{1/p^{\infty}}];\Z_p)\\
 &\simeq \THH(\sT_{\sO_\Cu};\Z_p)\otimes_{\THH(\bS[z_0^{1/p^{\infty}},z_1^{1/p^{\infty}}];\Z_p)}\bS[z_0^{1/p^{\infty}},z_1^{1/p^{\infty}}]^{\wedge}_{p}
\end{align*}
    which is compatible with $S^1$-action. For $i=0,1$, we have a $S^1$-equivariant equivalence $\THH(\bS[z_i^{1/p^{\infty}}];\Z_p)\simeq \bS[z_i^{1/p^{\infty}}]^{\wedge}_{p}$ \cite[Proposition 11.7]{BMS2}, and we obtain the claim.
\end{proof}


\begin{lemma}\label{TPTPbz0z1pinfty} 
The natural map $\bS[z_0,z_1] \to \bS[z_0^{1/p^{\infty}},z_1^{1/p^{\infty}}]$ induces an isomorphism
\begin{equation*}\label{piTPbSzz}
\pi_i \TP(\sT_{\sO_{\Cu}}/\bS[z_0^{1/p^{\infty}},z_1^{1/p^{\infty}}];\Z_p) \simeq \pi_i\TP(\sT/\bS[z_0,z_1];\Z_p) \otimes_{\mStw,\iota} \Ainf 
\end{equation*}
for any $i$.
\end{lemma}
\begin{proof}
The morphism $e_0:\bS[z] \to \bS[z_0,z_1]$ induces an isomorphism $\pi_i\TP(\sT/\bS[z];\Z_p) \otimes_{\mS,\theta_0} \mStw \simeq \pi_i \TP(\sT/\bS[z_0,z_1];\Z_p)$, and the composite map $\bS[z]\overset{e_0}{\to} \bS[z_0,z_1] \to \bS[z_0^{1/p^{\infty}},z_1^{1/p^{\infty}}]$ induces an isomorphism \[
\pi_i\TP(\sT/\bS[z];\Z_p) \otimes_{\mS,\phi} \Ainf \simeq \pi_i \TP(\sT_{\sO_{\Cu}}/\bS[z_0^{1/p^{\infty}},z_1^{1/p^{\infty}}];\Z_p).
\]
Since there is an equation of morphisms $\phi=\iota \circ \theta_0$, we obtain the claim.
\end{proof}

\subsection{$\tau$-action on $\TP(\sT_{\sO_{\Cu}};\Z_p)$} The construction in this section is inspired by H. Gao \cite{lettertokeiho}. Let $\tau \in \text{Gal}(L/\bigcup_{n=1}^{\infty} K(\zeta_{n}))$ be a
topological generator and $\tilde{\tau}\in G_K$ be the lift of $\tau$ satisfying $\tilde{\tau}(\pi^{1/p^{n}})=\zeta_n\pi^{1/p^{n}}$. Consider a map $\eta:\bS[z_0^{1/p^{\infty}},z_1^{1/p^{\infty}}] \to \sO_{\Cu}$ which sends $z_0^{1/p^{n}}$ to $\pi^{1/p^n}$ and $z_1^{1/p^n}$ to $\zeta_{n}\pi^{1/p^n}$.

We will use the following notations associated with a smooth proper $\sO_K$-linear category $\sT$.
\begin{itemize}
   \item $e^{\infty}_{0}:\TP(\sT_{\sO_{\Cu}}/\bS[z^{1/p^{\infty}}];\Z_p) \overset{\simeq}{\to}\TP(\sT_{\sO_{\Cu}}/\bS[z_0^{1/p^{\infty}},z_1^{1/p^{\infty}}];\Z_p)$ the morphism given by ${z^{1/p^{\infty}}\mapsto z_0^{1/p^{\infty}}}$
  
  \item $a:\TP(\sT/\bS[z];\Z_p) \to \TP(\sT_{\sO_{\Cu}}/\bS[z^{1/p^{\infty}}];\Z_p)$ the morphism given by the commutative diagram
  \[
  \xymatrix{
  \sO_K  \ar[r] & \sO_{\Cu} \\
  \bS[z] \ar[r] \ar[u] & \bS[z^{1/p^{\infty}}] \ar[u]
  }
  \]
  \item $b:\TP(\sT/\bS[z];\Z_p) \to \TP(\sT_{\sO_{\Cu}}/\bS[z_0^{1/p^{\infty}},z_1^{1/p^{\infty}}];\Z_p)$ the morphism given by commutative diagram
\[
  \xymatrix{
  \sO_K  \ar[r] & \sO_{\Cu} \\
  \bS[z] \ar[r]_-{z\mapsto z_0} \ar[u] & \bS[z_0^{1/p^{\infty}},z_1^{1/p^{\infty}}] \ar[u]_{\eta}
  }
  \]

  \item $c:\TP(\sT/\bS[z_0,z_1];\Z_p) \to \TP(\sT_{\sO_{\Cu}}/\bS[z_0^{1/p^{\infty}},z_1^{1/p^{\infty}}];\Z_p)$ the morphism given by commutative diagram
  \[
  \xymatrix{
  \sO_K \ar[r]& \sO_{\Cu}\\
\bS[z_0,z_1] \ar[u] \ar[r]_-{z_i\mapsto z_i}& \bS[z_0^{1/p^{\infty}},z_1^{1/p^{\infty}}] \ar[u]_{\eta}
  }
  \]  

  \item $\tilde{\tau}_1:\TP(\sT_{\sO_{\Cu}}/\bS[z^{1/p^{\infty}}];\Z_p)\to \TP(\sT_{\sO_{\Cu}}/\bS[z_0^{1/p^{\infty}},z_1^{1/p^{\infty}}];\Z_p)$ the morphism given by the commutative diagram
  \begin{equation}\label{tauonedia}
  \xymatrix{
  \sO_{\Cu} \ar[rr]^{\tilde{\tau}}& &\sO_{\Cu}\\
    \bS[z^{1/p^{\infty}}]\ar[u] \ar[rr]_-{z^{1/p^n}\mapsto z_1^{1/p^n}} &&\bS[z_0^{1/p^{\infty}},z_1^{1/p^{\infty}}] \ar[u]
  }
  \end{equation}
\end{itemize}

\begin{lemma}\label{4.9}
    The following diagram commutes
    \[
    \xymatrix{
   \TP(\sT_{\sO_{\Cu}};\Z_p) \ar[r]^{\tilde{\tau}}_{\simeq} \ar[d]_{}^{\simeq}& \TP(\sT_{\sO_{\Cu}};\Z_p)\ar[d]^{}_{\simeq} \\
   \TP(\sT_{\sO_{\Cu}}/\bS[z^{1/p^{\infty}}];\Z_p)\ar[r]_-{\tilde{\tau}_1} & \TP(\sT_{\sO_{\Cu}}/\bS[z_0^{1/p^{\infty}},z_1^{1/p^{\infty}}];\Z_p)
    }
    \]
\end{lemma}
\begin{proof}
    The diagram~\eqref{tauonedia} is compatible with $\bS$-ring spectrum structure, we obtain the claim.
\end{proof}
The following theorem, inspired by H. Gao \cite{lettertokeiho}, plays the most important role in the proof of Theorem~\ref{maininintro}.
\begin{prop}\label{4.10}
The following diagram commutes
\begin{equation}\label{twodiagram}
\xymatrix{
\TP(\sT_{\sO_{\Cu}}/\bS[z^{1/p^{\infty}}];\Z_p)\ar[r]^-{\tilde{\tau}_1}_-{\simeq} & \TP(\sT_{\sO_{\Cu}}/\bS[z_0^{1/p^{\infty}},z_1^{1/p^{\infty}}];\Z_p)& \TP(\sT_{\sO_{\Cu}}/\bS[z^{1/p^{\infty}}];\Z_p) \ar[l]_-{e_0^{\infty}}^-{\simeq} \\
\TP(\sT/\bS[z];\Z_p)\ar[r]_-{e_{1,\sT}} \ar[u]^-{a} & \TP(\sT/\bS[z_0,z_1];\Z_p) \ar[u]_{c} & \TP(\sT/\bS[z];\Z_p)\ar[l]^-{e_{0,\sT}} \ar[u]^-{a}
}
\end{equation}
\end{prop}
\begin{proof}
The diagram~\eqref{tauonedia} fits into the following commutative diagram:
\[
\xymatrix{
\sO_{\Cu} \ar[rr]^{\tilde{\tau}} & & \sO_{\Cu} & \\
&\bS[z^{1/p^{\infty}}] \ar[rr]^-{z^{1/p^{\infty}}\mapsto z_1^{1/p^{\infty}}} \ar[lu] && \bS[z_0^{1/p^{\infty}},z_1^{1/p^{\infty}}]\ar[lu]_-{\eta}\\
\sO_K \ar[uu] \ar[rr]|\hole^(0.3){\text{id}_{\sO_K}}& & \sO_K \ar[uu]|\hole&\\
&\bS[z] \ar[uu] \ar[rr]_{z\mapsto z_1} \ar[lu] & & \bS[z_0,z_1] \ar[uu] \ar[lu]
}
\]
Thus we see that the left square of \eqref{twodiagram} commutes. Look at the following diagram:
\[
\xymatrix{
\sO_{\Cu} \ar[rr]^{\text{id}_{\sO_{\Cu}}} & & \sO_{\Cu} & \\
&\bS[z^{1/p^{\infty}}] \ar[rr]^-{z^{1/p^{\infty}}\mapsto z_0^{1/p^{\infty}}} \ar[lu] && \bS[z_0^{1/p^{\infty}},z_1^{1/p^{\infty}}]\ar[lu]_-{\eta}\\
\sO_K \ar[uu] \ar[rr]|\hole^(0.3){\text{id}_{\sO_K}}& & \sO_K \ar[uu]|\hole&\\
&\bS[z] \ar[uu] \ar[rr]_{z\mapsto z_0} \ar[lu] & & \bS[z_0,z_1] \ar[uu] \ar[lu]
}
\]
then we see that the right
square of \eqref{twodiagram} commutes.
\end{proof}



Fix an integer $i >> 0$. By Proposition~\ref{TPTCkunneth}, $\pi_i \TP(\sT/\bS[z];\Z_p)$ is a finitely generated $\mS$-module. Choose a $\mS$-generator $\mm_1,\mm_2,....,\mm_n $ of $\pi_i \TP(\sT/\bS[z];\Z_p)$. We have an equation $\ol{\phi}=\phi \circ \varphi:\mS \to \Ainf$ (see the diagram \eqref{SAinfdiag}), combine isomorphism \eqref{TCnKisin} with Theorem~\ref{TPTCcomp} then we see that $a$ induces an isomorphism $\pi_i\TP(\sT/\bS[z];\Z_p)\otimes_{\mS,\phi}\Ainf \simeq \pi_i\TP(\sT_{\sO_{\Cu}}/\bS[z^{1/p^{\infty}}];\Z_p)$. For  $x\in \pi_i \TP(\sT/\bS[z];\Z_p)$, we write $\ol{x}$ for the image of $x$ under the composite map $\pi_i\TP(\sT/\bS[z];\Z_p) \overset{a}{\to} \pi_i\TP(\sT_{\sO_{\Cu}}/\bS[z^{1/p^{\infty}}];\Z_p)\overset{\simeq}{\leftarrow} \pi_i\TP(\sT_{\sO_{\Cu}};\Z_p)$, then $\ol{\mm_1},\ol{\mm_2},...,\ol{\mm_n}$ become a $\Ainf$-generator of $\pi_i\TP(\sT_{\sO_{\Cu}};\Z_p)$. Choose elements $a_{x1},a_{x2},...,a_{xn}$  of $\Ainf$ so that
\[
\tilde{\tau}(\ol{x}) = a_{x1}\ol{\mm_1} + a_{x2}\ol{\mm_2} + \cdots + a_{xn}\ol{\mm_n}.
\]
\begin{prop}
For any $x$, any choice of $\mS$-generators $\{\mm_1,\dots,\mm_n\}$, and any  choice of elements $a_{xl}\in\Ainf$, 
the element $a_{xl}$ is contained in $\mStw$ under the inclusion 
$\mStw \overset{\iota}{\hookrightarrow} \Ainf$.
\end{prop}
\begin{proof} Combine Proposition~\ref{4.10} with Lemma~\ref{4.9}, we see the following diagram commutes.  
\begin{equation}\label{threediagram}
\xymatrix{
\TP(\sT_{\sO_{\Cu}};\Z_p) \ar[r]^{\tilde{\tau}}_{\simeq} \ar[d]_{}^{\simeq}& \TP(\sT_{\sO_{\Cu}};\Z_p)\ar[d]^{}_{\simeq} \ar[dr]^-{\simeq} & \\
\TP(\sT_{\sO_{\Cu}}/\bS[z^{1/p^{\infty}}];\Z_p)\ar[r]^-{\tilde{\tau}_1}_-{\simeq} & \TP(\sT_{\sO_{\Cu}}/\bS[z_0^{1/p^{\infty}},z_1^{1/p^{\infty}}];\Z_p)& \TP(\sT_{\sO_{\Cu}}/\bS[z^{1/p^{\infty}}];\Z_p) \ar[l]_-{e_0^{\infty}}^-{\simeq} \\
\TP(\sT/\bS[z];\Z_p)\ar[r]_-{e_{1,\sT}} \ar[u]^-{a} & \TP(\sT/\bS[z_0,z_1];\Z_p) \ar[u]_{c} & \TP(\sT/\bS[z];\Z_p)\ar[l]^-{e_{0,\sT}} \ar[u]^-{a}
}
    \end{equation} We note that $\pi_i\TP(\sT_{\sO_{\Cu}};\Z_p) \to \pi_i\TP(\sT_{\sO_{\Cu}}/\bS[z_0^{1/p^{\infty}},z_1^{1/p^{\infty}}];\Z_p)$ is $\pi_0\TP(\sO_{\Cu};\Z_p)\simeq \Ainf$-linear. The morphism $c:\pi_i\TP(\sT/\bS[z_0,z_1];\Z_p) \to \pi_i\TP(\sT/\bS[z_0^{1/p^{\infty}},z_1^{1/p^{\infty}}];\Z_p)$ is given by $\pi_i\TP(\sT/\bS[z];\Z_p) \otimes_{\mS,\theta_0}\mStw \overset{\text{id}\otimes \iota}{\to} \pi_i\TP(\sT/\bS[z];\Z_p) \otimes_{\mS,\phi}\Ainf$ by Proposition~\ref{TPmSmS2} and Lemma~\ref{TPTPbz0z1pinfty}. Besides, by Proposition~\ref{TPmSmS2}, the morphism $e_{0,\sT}:\pi_i\TP(\sT/\bS[z];\Z_p)\to \TP(\sT/\bS[z_0,z_1];\Z_p)$ is given by $\pi_i\TP(\sT/\bS[z];\Z_p) \to \pi_i\TP(\sT/\bS[z];\Z_p) \otimes_{\mS,\theta_0} \mStw$, and $\{e_{0,\sT}(\mm_1),e_{0,\sT}(\mm_2),...,e_{0,\sT}(\mm_n)\}$ is a $\mStw$-generator of $\TP(\sT/\bS[z_0,z_1];\Z_p)$. Therefore, for $x \in \pi_i\TP(\sT/\bS[z];\Z_p)$, there are $a'_{x1},a'_{x2},...,a'_{xn} \in \mStw\overset{\iota}{\hookrightarrow}\Ainf$ such that 
    \begin{equation*}
    e_{1,\sT}(x)= \displaystyle\Sigma_{l=1}^{n} a_{xl}'e_{0,\sT}(\mm_l).
    \end{equation*}
    Look at the diagram, then we see an equation
    $\tilde{\tau}_1(a(x))=\Sigma a'_{xl}c(e_{0,\sT}(\mm_l))$, and we obtain the claim.
\end{proof}

\begin{remark}\label{Rem4.12}
We can apply the same procedure to $\tilde{\tau}^{-1}$ and we obtain that there are elements $b_{xj}$ of $\mStw $ so that $\tilde{\tau}^{-1}(\ol{x}) = b_{x1}\ol{\mm_1} + b_{x2}\ol{\mm_2} + \cdots + b_{xn}\ol{\mm_n}$, where instead of $\eta$ we use a morphism $\bS[z_0^{1/p^\infty},z_1^{1/p^\infty}] \to \sO_{\Cu}$ which sends $z_0^{1/p^n}$ to $\pi^{1/p^n}$ and $z_1^{1/p^n}$ to $\zeta_n^{-1}\pi^{1/p^n}$.
\end{remark}

 Let $\tilde{\tau}^*$ denote the dual action of $\tilde{\tau}$ on $\pi_i\TP(\sT_{\sO_{\Cu}};\Z_p)^{\vee}:=\Hom_{\Ainf}(\pi_i\TP(\sT_{\sO_{\Cu}};\Z_p),\Ainf)$. Choose a $\mS$-basis $\nn_1,\nn_2,...\nn_d$ of $\pi_i\TP(\sT/\bS[z];\Z_p)^{\vee}$. Combined \eqref{TCnKisin} with Theorem~\ref{TPTCcomp}, there is an isomorphism $\pi_i\TP(\sT/\bS[z];\Z_p)^\vee\otimes_{\mS,\phi}\Ainf \simeq \pi_i\TP(\sT_{\sO_{\Cu}};\Z_p)^\vee$. For $y\in \pi_i\TP(\sT_{\sO_{\Cu}};\Z_p)^{\vee}$, let $\ol{y}$ denote the image of $x$ under the inclusion $\pi_i\TP(\sT/\bS[z];\Z_p)^{\vee} \hookrightarrow \pi_i\TP(\sT/\bS[z];\Z_p)^{\vee}\otimes_{\mS,\phi}\Ainf \simeq \pi_i\TP(\sT_{\sO_{\Cu}};\Z_p)^{\vee}$.  Choose elements $b_{y1},b_{y2},...,b_{yd}$  of $\Ainf$ so that
 \[
 \tilde{\tau}^*(\ol{y})=b_{y1}\ol{\nn_{l}}+b_{y2}\ol{\nn_{2}}+\cdots + b_{yd}\ol{\nn_{d}}.
 \]
 Remark~\ref{Rem4.12} directly induces the following.

\begin{cor}\label{cor414} For any $y$, $b_{y1},b_{y2},...,b_{yd}$ are contained in $\mStw$ under the inclusion $\mStw \overset{\iota}{\hookrightarrow}\Ainf$.
\end{cor}

\subsection{$\tilde{\tau}$-action and crystalline representations} Let $\phi_1:\mS\to \Ainf$ be a $W$-linear map which sends $z$ to $[\epsilon][\pi^{\flat}]$. The following diagram commutes.
\[
\xymatrix{
\Ainf \ar[r]^{\tilde{\tau}} & \Ainf \\
\mS \ar[u]^{\phi} \ar[r]^{\theta_1} \ar[ru]^{\phi_1}& \mStw \ar[u]_{\iota}
}
\]
Let $E_1$ denote $\theta_1(E)$, $E_0$ denote $\theta_0(E)$ and $\xi_1$ denote $\tilde{\tau}(\xi)$. We note that $\mStw$ is an integral domain (see \cite[Proof of Lemma~2.3.2]{DuLiu}), and $\iota(E_1)=\xi_1$. Firstly, we prove the following Lemma.
\begin{lemma}\label{Lem415}
    As sub-rings of $\Ainf[\frac{1}{\xi_1}]$, there is an equation
    \[
    \mStw = \mStw[\frac{1}{E_1}] \cap \Ainf
    \]
    of sub-algebra $\Ainf[\frac{1}{\xi_1}]$, where we regard $\mStw[\frac{1}{E_1}]$ as a sub-ring of $\Ainf[\frac{1}{\xi_1}]$ via $\iota[\frac{1}{E_1}]:\mStw[\frac{1}{E_1}] \hookrightarrow \Ainf[\frac{1}{\xi_1}]$.
\end{lemma}
\begin{proof}

    It is suffice to show that $\mStw$ contains $\mStw[\frac{1}{E_1}] \cap \Ainf$. The map $\iota$ induces a morphism of $W$-algebra $\ol{\iota}:\mStw/(E_1) \to \Ainf/(\xi_1)$ and $\theta_1$ induces a morphism of $W$-algebra $\ol{\theta_1}:\mS/(E)\to \mStw/(E_1)$. The isomorphism \eqref{koukamS} induces an isomorphism of $\mS$-algebra $\mS\otimes_{\mS,\theta_0}\mStw \simeq \mS\otimes_{\mS,\theta_1}\mStw$ which sends $E_0$ to $E_1$. By \cite[Lemma~2.2.8 (2)]{DuLiu}, $\theta_0$ induces an isomorphism $\sO_K \simeq \mS/(E) \simeq \mStw/(E_0) \simeq \mStw/(E_1)$ of $W$-algebras. Thus we obtain that $\ol{\iota}$ is given by natural $W$-algebra $\sO_K \hookrightarrow \sO_{\Cu}$. We assume that there exists an element $x$ of $\mStw[\frac{1}{E_1}] \cap \Ainf \subset \Ainf[\frac{1}{\xi_1}]$ s.t. $x$ does not contained in $\mStw$. Since $x\in \mStw[\frac{1}{E_1}]$, there is the natural number $n\geq 1$ such that $E_1^n\cdot x$ is in $\mStw$ and $E_1^{n-1}\cdot x$ is not in $\mStw$. Since $\mStw$ is an integral domain,  $E_1^{n}\cdot x \not\in (E_1)\subset \mStw$. Thus the class $\ol{E_1^{n}\cdot x}$ is not zero in $\mStw/(E_1)$. On the other hand, since $x\in \Ainf$, $E_1^{n}\cdot x \in (\xi_1)\subset \Ainf$. Thus the class $E_1^{n}\cdot x$ is zero in $\Ainf/(\xi_1)$. This is contradictory to the fact that $\ol{\iota}:\mStw/(E_1) \to \Ainf/(\xi_1)$ is injective.
\end{proof}

Fix a smooth proper $\sO_K$-linear category $\sT$ and an integer $i \gg 0$.  
The finite free $\mS$-module $\pi_i\TCn(\sT/\bS[z];\Z_p)^{\vee}$ carries a natural $(\varphi,\widehat{G})$-module structure by Theorem~\ref{BKGKmodulethm} and Remark~\ref{remark4.2}. By functoriality of $\can_{(-)}$, we obtain the following commutative diagram:
\[
\xymatrix{
\pi_i\TCn(\sT/\bS[z];\Z_p)^{\vee}\otimes_{\mS}\Ainf 
  \ar[r]_{\simeq}^{\eqref{kenka1}} &
  \pi_i\TCn(\sT_{\sO_{\Cu}}/\bS[z^{1/p^{\infty}}];\Z_p)^{\vee} &
  \pi_i\TCn(\sT_{\sO_{\Cu}};\Z_p)^{\vee} \ar[l]_-{\simeq}  \\
\pi_i\TP(\sT/\bS[z];\Z_p)^{\vee}\otimes_{\mS}\Ainf 
  \ar[r]^{\simeq} \ar[u]_-{\can_{\sT}^{\vee}\otimes id_{\Ainf}} &
  \pi_i\TP(\sT_{\sO_{\Cu}}/\bS[z^{1/p^{\infty}}];\Z_p)^{\vee}
  \ar[u]_-{\can_{\sT_{\sO_{\Cu}/\bS[z^{1/p^{\infty}}]}}^{\vee}} &
  \pi_i\TP(\sT_{\sO_{\Cu}};\Z_p)^{\vee} 
  \ar[u]_-{\can_{\sT_{\sT_{\sO_{\Cu}}}}^{\vee}} \ar[l]_-{\simeq}
}
\]
Hence,
\begin{equation}\label{mujun2}
\can_{\sT}^{\vee}\otimes id_{\Ainf}\colon
\pi_i\TP(\sT/\bS[z];\Z_p)^{\vee}\otimes_{\mS,\phi}\Ainf
\longrightarrow
\pi_i\TCn(\sT/\bS[z];\Z_p)^{\vee}\otimes_{\mS,\phi}\Ainf
\end{equation}
is a $G_K$-equivariant morphism. 

We obtain the following commutative diagram:
\begin{equation}\label{commutativediagram}
\xymatrix{
\pi_i\TP(\sT/\bS[z];\Z_p)^{\vee} 
  \ar[rr]^{\can_{\sT}^{\vee}}\ar[d] &&
  \pi_i\TCn(\sT/\bS[z];\Z_p)^{\vee} \ar[d] \\
\pi_i\TP(\sT/\bS[z];\Z_p)^{\vee}\otimes_{\mS,\phi}\Ainf 
  \ar[rr]^-{\can_{\sT}^{\vee}\otimes id_{\Ainf}} \ar[d]_{\ol{\tau}^*} &&
  \pi_i\TCn(\sT/\bS[z];\Z_p)^{\vee}\otimes_{\mS,\phi}\Ainf 
  \ar[d]_{\ol{\tau}^*} \\
\pi_i\TP(\sT/\bS[z];\Z_p)^{\vee}\otimes_{\mS,\phi}\Ainf
  \ar[rr]^-{\can_{\sT}^{\vee}\otimes id_{\Ainf}} &&
  \pi_i\TCn(\sT/\bS[z];\Z_p)^{\vee}\otimes_{\mS,\phi}\Ainf
}
\end{equation}
Note that $\can_{\sT}^{\vee}$ is $\mS$-linear, and $\can_{\sT}^{\vee}\otimes id_{\Ainf}$ is $\Ainf$-linear.

Choose an $\mS$-basis $\nn_1,\dots,\nn_d$ of 
$\pi_i\TCn(\sT/\bS[z];\Z_p)^{\vee}$.  
For $x \in \pi_i\TCn(\sT/\bS[z];\Z_p)^{\vee}$, let $\overline{x}$ denote its image under the inclusion
\[
\pi_i\TCn(\sT/\bS[z];\Z_p)^{\vee}
\hookrightarrow
\pi_i\TCn(\sT/\bS[z];\Z_p)^{\vee}\otimes_{\mS,\phi}\Ainf.
\]
Since $\{\ol{\nn}_1,\dots,\ol{\nn}_d\}$ is an $\Ainf$-basis of 
$\pi_i\TCn(\sT/\bS[z];\Z_p)^{\vee}\otimes_{\mS,\phi}\Ainf$,  
for each $x \in \pi_i\TCn(\sT/\bS[z];\Z_p)^{\vee}$ there exist elements 
$a_{x1},\dots,a_{xd}\in\Ainf$ such that
\[
\ol{\tau}^*(\ol{x})
= a_{x1}\ol{\nn}_1 + a_{x2}\ol{\nn}_2 + \cdots + a_{xd}\ol{\nn}_d.
\]

\begin{prop}
For any $x$, the elements $a_{x1},a_{x2},\dots,a_{xd}$ lie in $\mS^{(2)}$ under the inclusion 
$\mS^{(2)} \overset{\iota}{\hookrightarrow} \Ainf$.
\end{prop}

\begin{proof}\label{prop416}
After inverting $E$, both $\can_{\sT}^{\vee}$ and the inclusion
\[
\pi_i\TCn(\sT/\bS[z];\Z_p)^{\vee}
\hookrightarrow
\pi_i\TCn(\sT/\bS[z];\Z_p)^{\vee}
\]
become isomorphisms (see \eqref{canisom}).  
Hence, for any $x\in \pi_i\TCn(\sT/\bS[z];\Z_p)^{\vee}$, 
there exists $y\in \pi_i\TP(\sT/\bS[z];\Z_p)^{\vee}$ such that 
\begin{equation}\label{yequalEmx}
f(y)=E^m\cdot x \qquad
\text{in } \pi_i\TCn(\sT/\bS[z];\Z_p)^{\vee}.    
\end{equation}

Choose an $\mS$-basis $\nn_1,\dots,\nn_d$ of $\pi_i\TP(\sT/\bS[z];\Z_p)^{\vee}$.  
There exist elements $c_{jl}\in\mS$ (for $1\le j\le d$, $1\le l\le n$) such that
\begin{equation}\label{nnjcjl}
\can_{\sT}^{\vee}(\nn_j)=c_{j1}\mm_1+c_{j2}\mm_2+\cdots+c_{jn}\mm_n    
\end{equation}
in $\pi_i\TCn(\sT/\bS[z];\Z_p)^{\vee}$.  
Since the upper square of \eqref{commutativediagram} is $\mS$-linear, we have
\begin{equation}\label{olnnjcjl}
(\can_{\sT}^{\vee}\otimes id_{\Ainf})(\ol{\nn}_j)
=\phi(c_{j1})\ol{\mm}_1+\phi(c_{j2})\ol{\mm}_2+\cdots+\phi(c_{jn})\ol{\mm}_n    
\end{equation}
in $\pi_i\TCn(\sT/\bS[z];\Z_p)^{\vee}\otimes_{\mS,\phi}\Ainf$.  
Choose elements $b_{y1},\dots,b_{yd}\in\Ainf$ such that
\begin{equation}\label{hozon}
\tilde{\tau}^*(\ol{y})
=b_{y1}\ol{\nn}_1+b_{y2}\ol{\nn}_2+\cdots+b_{yd}\ol{\nn}_d     
\end{equation}
in $\pi_i\TP(\sT/\bS[z];\Z_p)^{\vee}$.  
By commutativity of \eqref{commutativediagram}, we obtain
\begin{eqnarray*}
\ol{\tau}^*(\ol{E^m\cdot x})
&=& (\can_{\sT}^{\vee}\otimes id_{\Ainf})(\ol{\tau}^*(\ol{y}))\\
&\overset{\eqref{hozon}}{=}&
(\can_{\sT}^{\vee}\otimes id_{\Ainf})\bigl(b_{y1}\ol{\nn}_1+b_{y2}\ol{\nn}_2+\cdots+b_{yd}\ol{\nn}_d\bigr)\\
&\overset{\eqref{olnnjcjl}}{=}&
\Bigl(\sum_{j=1}^{d}b_{yj}\phi(c_{j1})\Bigr)\ol{\mm}_1
+\Bigl(\sum_{j=1}^{d}b_{yj}\phi(c_{j2})\Bigr)\ol{\mm}_2
+\cdots
+\Bigl(\sum_{j=1}^{d}b_{yj}\phi(c_{jn})\Bigr)\ol{\mm}_n.
\end{eqnarray*}

Since $\ol{\tau}^*(\ol{E^m\cdot x}) = E_1^m\cdot \ol{\tau}^*(\ol{x})$, we have
\[
E_1^m\cdot a_{xl} = \sum_{j=1}^{d} b_{yj}\phi(c_{jl})
\]
in $\Ainf$.  
By Diagram~\ref{mSmS2Ainf} and Corollary~\ref{cor414}, the right-hand side lies in $\mS^{(2)}$.  
Hence $a_{xl}\in \mS^{(2)}[1/E_1]\cap \Ainf = \mS^{(2)}$.
\end{proof}

\begin{cor}\label{cor417}
The $p$-adic representation $T_{\Ainf}(\pi_i\TCn(\sT/\bS[z];\Z_p)^{\vee})$ is a $\Z_p$-lattice of a crystalline representation.
\end{cor}
\begin{proof}
By \cite[Corollary 3.3.4]{DuLiu}, it is suffice to show that $\tilde{\tau}^*(\pi_i\TCn(\sT/\bS[z];\Z_p)^{\vee})$ is contained in $\pi_i\TCn(\sT/\bS[z];\Z_p)^{\vee}\otimes_{\mS,\theta_0}\mStw$. By Proposition \ref{prop416}, we obtain the claim.
\end{proof}

\begin{remark}
    We expect a more geometric proof of the corollary by showing that, for each $i$, the group $\pi_i\,\TC^{-}(\sT/\mathbb{S}[z];\Z_p)$ underlies a prismatic $F$-crystal. In particular, the associated descent datum $\pi_i\,\TC^{-}(\sT/\mathbb{S}[z];\Z_p)\;\rightrightarrows\;\pi_i\,\TC^{-}(\sT/\mathbb{S}[z];\Z_p)\otimes_{\mS}\mS^{(2)}\;\rightthreearrow  \;\pi_i\,\TC^{-}(\sT/\mathbb{S}[z];\Z_p)\otimes_{\mS}\mS^{(3)}$ should be induced by the canonical maps:
\[
\pi_i\,\TC^{-}(\sT/\mathbb{S}[z];\Z_p)
\;\rightrightarrows\;
\pi_i\,\TC^{-}(\sT/\mathbb{S}[z_0,z_1];\Z_p)
\;\rightthreearrow\;
\pi_i\,\TC^{-}(\sT/\mathbb{S}[z_0,z_1,z_2];\Z_p).
\]
However, we do not currently know whether the natural base–change map
\[
\pi_i\,\TC^{-}(\sT/\mathbb{S}[z];\Z_p)\otimes_{\mS}\mS^{(3)}
\;\xrightarrow{\;\sim\;}\;
\pi_i\,\TC^{-}(\sT/\mathbb{S}[z_0,z_1,z_2];\Z_p)
\]
is an isomorphism; this hinges on the as–yet unclear structure of $\pi_*\TC^{-}(\sO_K/\mathbb{S}[z_0,z_1,z_2];\Z_p)$.

\end{remark}

\subsection{The proof of Main Theorems} The following theorem is the summary of section 3 and section 4.

\begin{thm}\label{truemainwithproof}
Let $\sT$ be a smooth proper $\sO_K$-linear category. Then there exists an integer $n\ge 0$ such that, for all $i\ge n$, the following statements hold:
\begin{itemize}
 \item[(1)] Breuil-Kisin module $\pi_i \TCn(\sT/\bS[z];\Z_p)^{\vee}$ has a Breuil-Kisin $G_K$-module structure in the sense of \cite{Gao}.

 \item[(2)] Breuil-Kisin module $\pi_i \TCn(\sT/\bS[z];\Z_p)^{\vee}$ has a $(\varphi,\hat{G})$-module structure in the sense of \cite{DuLiu}.

 \item[(3)] The $\Z_p[G_K]$-module $T_{\Ainf}(\pi_i \TCn(\sT/\bS[z];\Z_p)^{\vee})$ is a $\Z_p$-lattice of a crystalline representation.

 \item[(4)] If $\sT_{\Cu}$ admits a geometric realization, then there is a $G_K$-equivariant isomorphism 
 \[
 T_{\Ainf}(\pi_i \TCn(\sT/\bS[z];\Z_p)^{\vee}) \simeq \pi_i\LK K(\sT_\Cu)^{\vee}
 \]
 of $\Z_p$-modules.
\end{itemize}
\end{thm}
\begin{proof}
(1) follows from Theorem~\ref{BKGKmodulethm}. (2) follows from Remark~\ref{remark4.2}. (3) follows from Corollary~\ref{cor417}. (4) follows from Theorem~\ref{mainwithproof}.
\end{proof}
\begin{thm}[Main Theorem]\label{mainmainmain1}
     Let $\sT$ be a smooth proper $\sO_K$-linear category. If $\sT_{\Cu}$ admits a geometric realization, then Conjecture~\ref{NCcryconj} holds for $\sT$, i.e, there is an isomorphism of $\Bcry$-module:
\[
\pi_i\TP(\sT_k;\Z_p)\otimes_W \Bcry \simeq \pi_i \LK K(\sT_\sC)\otimes_{\Z_p}\Bcry
\]
which is compatible with $G_K$-action and Frobenius endomorphism. 
\end{thm}
\begin{proof}
    Firstly, we will prove the claim when $i$ is large enough. By Remark~\ref{rem38}, we have a $G_K$-equivariant isomorphism 
    \[
    \pi_i \LK K(\sT_{\Cu})^{\vee\vee} \simeq T_{\Ainf}(\pi_i\TCn(\sT/\bS[z];\Z_p)^{\vee\vee}),
    \]
    thus we have an isomorphism
    \[
    \pi_i \LK K(\sT_{\Cu})^{\vee\vee} \otimes_{\Z_p} \Bcry \simeq \pi_i\TCn(\sT/\bS[z];\Z_p)^{\vee\vee}\otimes_{\mS}\Bcry
    \]
    which is $G_K$-equivariant and compatible with Frobenius endomorphism, and there is the identification of rational Dieudonn\'{e} modules (see \cite[Remark 4.5]{BMS1})
    \[
    D_{\text{crys}}(\pi_i \LK K(\sT_{\Cu})^{\vee\vee}\otimes_{\Z_p}\Q_p) = \pi_i\TCn(\sT/\bS[z];\Z_p)^{\vee\vee} \otimes_{\mS,\tilde{\phi}}W[\frac{1}{p}].
    \] By Theorem~\ref{TCnTPcomp3}, we have an isomorphism $\pi_i\TCn(\sT/\bS[z];\Z_p)^{\vee\vee} \otimes_{\mS,\tilde{\phi}}W[\frac{1}{p}] \simeq \pi_i\TP(\sT_{k};\Z_p)[\frac{1}{p}]$. Thus we obtain the claim.

 Now we prove the general case. For any $i\in\Z$, since $\LK K(\sT)$ is $2$-priodic, there is an isomorphism
 \[
 \pi_i \LK K(\sT) \otimes_{\pi_0 \LK K(\Cu)} \pi_2 \LK K(\Cu) \simeq \pi_{i+2} \LK K(\sT),
 \]
 and we have an isomorphism $\pi_2 \LK K(\Cu) \simeq \Z_p(1)$ (see \cite{Thomason}), we obtain the isomorphism 
 \[
 \pi_i \LK K(\sT) (1) \simeq\pi_{i+2} \LK K(\sT).
 \]
 Similarly, there is an isomorphism
 \[
 \pi_i \TP(\sT_k;\Z_p)[\frac{1}{p}] \otimes_{\pi_0 \TP(k;\Z_p)[\frac{1}{p}]} \pi_2 \TP(k;\Z_p)[\frac{1}{p}] \simeq \pi_{i+2} \TP(\sT_k;\Z_p)[\frac{1}{p}],
 \] 
 and we have an isomorphism $\pi_2 \TP(k;\Z_p)[\frac{1}{p}] \simeq W[\frac{1}{p}](1)$ (see \cite{BMS2}), where the twist refers to twisting the Frobenius, we obtain an isomorphism
 \[
\pi_i \TP(\sT_k;\Z_p)[\frac{1}{p}](1) \simeq \pi_{i+2} \TP(\sT_k;\Z_p)[\frac{1}{p}].
 \] 
Since we already proved the claim when $i$ is large enough, we obtain the claim. 
 \end{proof}



\section*{Acknowledgements}
\thispagestyle{empty}
The author would like to thank Federico Binda, Lars Hesselholt, Buntaro Kakinoki, Shunsuke Kano, Hyungseop Kim and Hiroyasu Miyazaki for helpful discussions related to this subject. The author would also like to thank Ryomei Iwasa for comments on an earlier draft and Alexender Petrov for useful comments on a draft, and Tasuki Kinjo for helpful discussions about derived schemes, Isamu Iwanari, Atsushi Takahashi and Shinnosuke Okawa for helpful discussions about non-commutative algebraic geometry. The author is deeply grateful to Hui Gao for helpful discussions about $(\varphi,\hat{G})$-modules and for sharing his ideas about Theorem~\ref{main} and Theorem~\ref{maininintro} with us. The author is also grateful to the anonymous referee for a careful reading of the manuscript and for many insightful comments and suggestions, which greatly improved the paper.

\section*{Funding}
\thispagestyle{empty} This work was supported by Kakenhi Grants 21H04994.


\bibliography{bib}
\bibliographystyle{junsrt}


\end{document}